\documentclass[a4paper,11pt]{article}

\usepackage{amsmath,amsthm}
\usepackage{amssymb}
\usepackage{breqn}
\usepackage{enumerate}
\usepackage{graphicx}
\usepackage{bm}
\usepackage{bbm}
\usepackage[affil-it]{authblk}
\usepackage{tabu}
\usepackage{bold-extra}
\usepackage[hang,flushmargin]{footmisc}
\usepackage[driverfallback=dvipdfm]{hyperref}
\usepackage{mathtools}
\usepackage{relsize}
\usepackage{scalerel}
\usepackage{xcolor}
\usepackage{pagecolor}
\usepackage{titlesec}
\usepackage{apptools}
\usepackage{appendix}
\usepackage{dsfont}
\newtheorem{theorem}{Theorem}
\newtheorem{corollary}[theorem]{Corollary}
\newtheorem{lemma}[theorem]{Lemma}
\newtheorem{proposition}[theorem]{Proposition}

\newtheorem{claim}[theorem]{Claim}

\theoremstyle{definition}

\titleformat{\section}[hang]{\scshape\large\bfseries\filcenter}{\S\thesection}{4pt}{}
\titleformat{\subsection}[hang]{\scshape\bfseries}{\thesubsection.}{4pt}{}

\allowdisplaybreaks

\def\dj{d\kern-0.4em\char"16\kern-0.1em}
\def\Dj{\mbox{\raise0.3ex\hbox{-}\kern-0.4em D}}
\def\djcaps{\mbox{\raise0.17ex\hbox{-}\kern-0.335em d}}

\newcommand\id{\mathbbm{1}}	
\newcommand{\tss}[1]{\textsuperscript{#1}}
\newcommand{\on}[1]{
	\operatorname{#1}
}

\newcommand{\bigconv}[1]{
	\mathbf{C}_{#1}
}

\newcommand{\tdt}{\times\cdots\times}

\newcommand{\tightoverset}[2]{
  \mathop{#2}\limits^{\vbox to -.5ex{\kern-1.15ex\hbox{$#1$}\vss}}}

\newcommand\blfootnote[1]{%
  \begingroup
  \renewcommand\thefootnote{}\footnote{#1}%
  \addtocounter{footnote}{-1}%
  \endgroup
}

\renewenvironment{thebibliography}[1]
{
  \begin{oldthebibliography}{#1}
    \setlength{\itemsep}{0em  plus 0.3ex}
    \setlength{\parskip}{0em}
}
{
  \end{oldthebibliography}
}

\newcommand\ssk[1]{
	\substack{#1}
}

\newcommand\ex{\mathop{\mathbb{E}}}

\newcommand{\exx}{
  \mathop{
    \mathchoice{\vcenter{\hbox{\larger[4]$\mathbb{E}$}}}
               {\kern0pt\mathbb{E}}
               {\kern0pt\mathbb{E}}
               {\kern0pt\mathbb{E}}
  }\displaylimits
}

\makeatletter
\newcommand*\bcdot{\mathpalette\bigcdot@{0.5}}
\newcommand*\bigcdot@[2]{\mathbin{\vcenter{\hbox{\scalebox{#2}{$\m@th#1\bullet$}}}}}
\makeatother

\makeatletter
\def\blfootnote{\gdef\@thefnmark{}\@footnotetext}
\makeatother

\begin{document}

\begin{center}\Large\noindent{\bfseries{\scshape M\"obius function is strongly orthogonal to polynomial phases over $\mathbb{F}_p[t]$}}\\[24pt]\normalsize\noindent{\scshape Luka Mili\'cevi\'c\tss{\dag} and \v{Z}arko Ran\djcaps elovi\'c\tss{\ddag}}
\end{center}
\blfootnote{\noindent\dag\ Mathematical Institute of the Serbian Academy of Sciences and Arts\\\phantom{\ddag\ }Email: luka.milicevic@turing.mi.sanu.ac.rs\\
\noindent\ddag\ Mathematical Institute of the Serbian Academy of Sciences and Arts\\
\phantom{\ddag\ }Email: zarko.randjelovic@turing.mi.sanu.ac.rs}

\footnotesize
\begin{changemargin}{1in}{1in}
\centerline{\sc{\textbf{Abstract}}}
\phantom{a}\hspace{12pt}~In this paper, we prove power-saving bounds for the corelation of the M\"obius function with polynomial phases of degree $k$ in function fields $\mathbb{F}_p[t]$, when $p > k$. The proof relies on a new approximation result for phases of biased multilinear forms and the recently established strong bounds for the problem of finding bounded codimension varieties inside the dense ones. Along the way, we also obtain polynomial bounds in the inverse theorem for Gowers uniformity norms in the special case of polynomial phases in finite vector spaces.
\end{changemargin}
\normalsize

\section{Introduction}

\hspace{12pt} Let $p$ be a fixed prime. The M\"obius function over the polynomial ring $\mathbb{F}_p[t]$ is defined as 
\[\mu(f) =  \begin{cases}
    (-1)^k,& \text{where }k\text{ is the number of monic irreducible factors of }f,\text{ when }f\text{ is square-free,}\\
    0, &\text{ otherwise,}
\end{cases}\]
at a polynomial $f(t)$. We write $A_n$ and $G_n$ for the sets of all monic degree $n$ polynomials in $\mathbb{F}_p[t]$ and all degree at most $n-1$ polynomials respectively. Viewing $\mathbb{F}_p[t]$ as a vector space over $\mathbb{F}_p$, the set $G_n$ is a subspace of $\mathbb{F}_p[t]$ of dimension $n$ and $A_n$ is a coset of $G_n$.\\

Throughout the paper, $\chi : \mathbb{F}_p \to \mathbb{C}$ is a non-trivial additive character, say $\chi(a) = \exp(\frac{2 \pi i a}{p})$, though the choice does not affect the results or definitions in the paper. Our main aim in this paper is to prove that the M\"obius function is strongly orthogonal to polynomial phases, with power-saving bounds.

\begin{theorem}\label{main_thm}
    Let $k \in \mathbb{N}$ and suppose that $p > k$. There exist
 constants $\varepsilon_k$ and $D_k$ depending only on $k$ such that 
    \[
    \frac{1}{p^n}\Big|\sum_{f \in G_n} \mu(f)\chi(Q(f))\Big| \leq  D_k p^{-\varepsilon_k n}
    \]
    holds for any degree $k$ polynomial $Q$ on the vector space $G_n$.
\end{theorem}

Along the way we prove a new result on the structure of multilinear forms whose values are not equidistributed. Recall that a map $\alpha : V_1 \tdt V_k \to \mathbb{F}_p$, where $V_1, \dots, V_k$ are finite-dimensional vector spaces over $\mathbb{F}_p$ is a \textit{multilinear form} if it is linear in each of its variables. A very useful measure of its structure is the \textit{bias} (and closely related analytic rank of Gowers and Wolf introduced in~\cite{GowWolf}), defined as the average of its phase 

\[\on{bias}\alpha = \frac{1}{|V_1|\cdots |V_k|} \sum_{x_1 \in V_1, \dots, x_k \in V_k} \chi \circ \alpha(x_1, \dots, x_k).\]

In this paper, we obtain the following approximation theorem for multilinear forms of large bias, saying that the phase of any such form is close in $L^2$ norm to a linear combination of phases of lower-order forms. To state it more precisely, we recall the notion of a \textit{multiaffine form}, which is a map from $V_1 \tdt V_k $ to $\mathbb{F}_p$ that is affine in each of its variables. We say that it has a \textit{vanishing multilinear part} if it does not have monomials of degree $k$. We may now state the approximation theorem.

\begin{theorem}\label{biased_ml_form_approx-intro}
    Let $\alpha : V_1 \tdt V_k \to \mathbb{F}_p$ be a multilinear form with $\on{bias}\alpha \geq c$. Let $\varepsilon > 0$. Then there exist
    $m \leq 2^{25}\varepsilon^{-14}c^{-14}$, coefficients $c_1, \dots, c_m \in \mathbb{C}$ such that $\sum_{i \in [m]} |c_i| \leq c^{-1}$, and multiaffine forms $\lambda_1, \dots, \lambda_m : V_1 \tdt V_k \to \mathbb{F}_p$ with vanishing multilinear part such that 
    \[\Big\| \chi \circ \alpha - \sum_{i \in [m]} c_i \chi \circ \lambda_i\Big\|_{L^2} \leq \varepsilon.\]
\end{theorem}

The bound on $m$, which is the number of multilinear forms in the linear combination above, is optimal up to constant in the degree of $(2\varepsilon{^-1}c^{-1})$, at least in the case of interest when $\varepsilon \leq c$.\\

Along the way, we also establish polynomial bounds in the inverse theorem for Gowers uniformity norms in the important special case of polynomial phases in finite vector spaces.

\begin{corollary}
    Let $V = \mathbb{F}_p^n$, $Q : V \to \mathbb{F}_p$ be a polynomial of degree $k$ and assume $k < p$. If $\|\chi \circ Q\|_{\mathsf{U}^k} \geq c$ then there exists a polynomial $P : V \to \mathbb{F}_p$ of degree at most $k-1$ such that
    \[\Big|\sum_{x \in V} \chi(Q(x))\overline{\chi(P(x))}\Big| \geq (c/2)^{O_k(1)} |V|.\]
\end{corollary}

\vspace{\baselineskip}

\noindent\textbf{Brief history of the problem.} Results in the spirit of Theorem~\ref{main_thm} were first considered over integers. As a part of their programme for counting the number of solutions of simultaneous linear systems of equations in primes, Green and Tao~\cite{GreenTaoMob1, GreenTaoMob2} first proved that the M\"obius function is orthogonal to nilsequences (which are a more general class of functions, playing the role of polynomial phases in additive combinatorics in the integer setting). Namely, they showed that for all $A > 0$ and polynomial nilsequence $n \mapsto F(g(n)\Gamma) $
\[\frac{1}{N}\sum_{n \in [N]}\mu(n)F(g(n)\Gamma) = O(\log^{-A} N),\]
with an ineffective implicit constant, due to ineffective bounds on Siegel zeroes. More classically, $\sum_{n \in [N]} \mu(n)= o(N)$ is equivalent to the prime number theorem and correlations with linear phases were studied by Davenport~\cite{Davenport}.\\
 
Over the function fields, the first quadratic result was obtained by Beinvenu and Le~\cite{BL}, who proved a quasipolynomial bound
\[\frac{1}{p^n}\Big|\sum_{f \in G_n} \mu(f)\chi(Q(f))\Big| \leq  O(p^{-n^c}),\]
for an absolute constant $c > 0$, which relied on the bilinear Bogolyubov argument~\cite{BienvenuLeBog, BilinearBog, HosseiniLovett}, leading to bounds of the form above.\\
\indent More recently, Meilin~\cite{Meilin} obtained orthogonality of the M\"obius function with higher degree polynomials with a bound of similar shape, which stems from the use of the so-called partition vs. analytic rank problem~\cite{BhowLov, GreenTaoPolys, Janzer2, LukaPrank, MoshZhuRank}, which we shall discuss slightly later. In our work, we evade both the higher-order Bogolyubov argument and the use of the partition vs. analytic rank problem.\\

We also note results of Sawin and Shusterman~\cite{SawinShusterman} who proved $\sum_{f \in G_n} \mu(Q(f)) = o(q^n)$ in $\mathbb{F}_q[t]$, with power-saving bounds. However, their proof works only over fields $\mathbb{F}_q$ for sufficiently large powers $q$ of a fixed prime $p$ and, in particular, it does not cover the prime fields. Our proof works over any finite field, but we opted to present it in the case of prime fields for the sake of simplicity of notation.\\

\noindent\textbf{Equidistribution theory of multilinear forms.} A significant objective of the higher order Fourier analysis and additive combinatorics is the study of multilinear forms that are not quasirandom. The simplest way of defining quasirandomness in this context is to study the distribution of values of a multilinear form $\alpha : V_1 \tdt V_k \to \mathbb{F}_p$. It turns out that 0 is taken most often, and that non-zero values are taken an equal number of times. In particular, the bias of $\alpha$ determines the distribution of values of $\alpha$ completely. \\

There are also simple algebraic obstructions preventing a uniform distribution of values, coming from products of multilinear form of lower order. Namely, it is an easy exercise to show that if $\alpha(x_1, \dots, x_k)$ can be written as $\beta(x_I) \gamma(x_{[k] \setminus I})$ for a multilinear form $\beta$ depending on variables whose indices lie in a set $I\subseteq [k]$ and $\gamma$ is another multilinear form depending on the rest if variables, then $\on{bias} \alpha \geq 1/p$. Remarkably, this turns out to be essentially the only way having a large bias. Formally, we define the \textit{partition rank}~\cite{Naslund} of a multilinear form $\alpha$, denoted $\on{prank} \alpha$, as the least number $r$ such that $\alpha$ can be expressed as a sum of $r$ multilinear forms that factorize in the way above. It is not hard to show that $\on{prank} \alpha = r$ implies $\on{bias} \alpha \geq p^{-r}$. The following theorem shows the opposite implication.

\begin{theorem}\label{prankthm-intro}
    For each $k \in \mathbb{N}$ and $c > 0$ there exists a positive integer $K = K(k, c)$ for which the following holds. Let $\alpha : V_1\tdt V_k \to \mathbb{F}_p$ be a multilinear form such that $\on{bias} \alpha \geq c$. Then $\on{prank}\alpha \leq K$.
\end{theorem}

The problem of relating these two quantities, bias and the partition rank, has since become known as the \textit{partition vs.\ analytic rank problem}. This result was obtained first for polynomials rather than multilinear forms by Green and Tao~\cite{GreenTaoPolys}. Kaufman and Lovett~\cite{KaufmanLovett} refined their approach and Bhowmick and Lovett~\cite{BhowLov} generalized it to the multilinear setting. The bounds on $K$ in terms of $c$ were of Ackermannian-type and significant quantitative improvements were obtained since~\cite{Janzer2, LukaPrank, MoshZhuRank}. However, the key open problem concerning biased multilinear forms, stated several times~\cite{AdipKazhZie, KazhZie, LamZie, Lovett}, remains to show that the partition rank is linear in terms of $\log_p c^{-1}$. For example, in order to keep power-saving bounds in our main result (Theorem~\ref{main_thm}), any application of Theorem~\ref{prankthm-intro} would require such a linear dependence.\\

Recently, in~\cite{LukaOptimalWeak}, the first-named author obtained a weak structural result for biased multilinear forms with linear bounds (see Theorem~\ref{weakinverse}), which informally speaking, identifies half of the lower-order forms required in an optimal partition rank decomposition. One of the key takeaways of our paper is that, even when understanding of structural properties of multilinear forms is needed, one can work directly with the bias of a multilinear form instead of passing to the partition rank, and use the mentioned weak structural result and approximation theorem (Theorem~\ref{biased_ml_form_approx-intro}) for biased multilinear forms in place of the partition vs.\ analytic rank problem.\\

\subsection{Proof overview}

The proof consists of 4 main steps. Overall, we induct over the degree $k$ of polynomial $Q$. During the proof, we study the symmetric multilinear form $\mathcal{L}_Q$, which is naturally associated to $Q$. The overall strategy is to show that if $\frac{1}{p^n}\Big|\sum_{f \in G_n} \mu(f)\chi(Q(f))\Big|$ is large, then the bias of $\mathcal{L}_Q$ is large as well. Once we obtain that conclusion, we may replace $Q$ by a polynomial of lower order $\tilde{Q}$ and still have a significant corelation $\frac{1}{p^n}\Big|\sum_{f \in G_n} \mu(f)\chi(\tilde{Q}(f))\Big|$. The induction hypothesis will then give a contradiction. Thus, showing that  $\mathcal{L}_Q$ is biased is key, and we now give an outline of that argument.

\begin{itemize}
    \item[\textbf{Step 1.}] \textit{Number-theoretic reduction to the study of multilinear forms.} Following Bienvenu and L\^{e}, who adapted an argument of Green and Tao to the function field setting, we use Vaughan's identity to reduce the problem to a study of a certain multilinear form. Namely, we deduce that the multilinear map in $2k$ variables $a_1, \dots, a_k, x_1, \dots, x_k$, defined by $\sum_{\pi \in \on{Sym}_k}\mathcal{L}_Q(a_1 x_{\pi(1)}, \dots, a_k x_{\pi(k)})$ is biased. The action of the symmetric group $\on{Sym}_k$ naturally appears as a consequence of an application of the Gowers-Cauchy-Schwarz inequality.
    \item[\textbf{Step 2.}] \textit{Finding a low-codimensional variety of biased multilinear lower-order forms.} We now depart from the previous works on this problem, and consider variables $a_i$ and $x_i$ separately. Namely, we show that there exists a low-codimensional variety whose points $(a_1, \dots, a_k)$ have the property that the form
    \begin{equation}\label{biasedmapsintrovariety}
        (x_1, \dots, x_k) \mapsto \sum_{\pi \in \on{Sym}_k}\mathcal{L}_Q(a_1 x_{\pi(1)}, \dots, a_k x_{\pi(k)})
    \end{equation}
    has large bias. This part of the proof relies on the above-mentioned recent weak structural result for biased multilinear forms~\cite{LukaOptimalWeak}, showing that dense multilinear varieties contain low-codimensional varieties with strong bounds. The argument also involves a dependent random choice argument, as well as consideration of directional convolutions.
    \item[\textbf{Step 3.}] \textit{Deducing that $\mathcal{L}_Q$ has large bias.} The structure of a variety from the previous step allows us to use the Chevalley-Warning theorem to obtain a special basis of $G_n$, with respect to which the multilinear form $\mathcal{L}_Q$ has particularly nice properties and we may deduce that the bias of $\mathcal{Q}$ is large by carefully combining appropriate restrictions of the form~\eqref{biasedmapsintrovariety}.\\
    \phantom{\,}\hspace{6pt} This argument is inspired by a similar one in the quadratic setting by Bienvenu and L\^{e}. However, their argument used a less efficient quasirandomness argument in place of the Chevalley-Warning theorem and was specialized to the case of matrices, where rank is the right notion of structure. The natural generalization of their argument would have to use partition rank of multilinear forms, but as discussion surrounding Theorem~\ref{prankthm-intro} suggests, to get strong bounds, we have to avoid this and work with bias directly. The structure of the symmetric group $\on{Sym}_k$ for $k > 2$ also complicates the analysis further.
    \item[\textbf{Step 4.}] \textit{Solving the biased polynomial special case.} In this step, we use the approximation theorem for biased multilinear forms (Theorem~\ref{biased_ml_form_approx-intro}) to replace $Q$ by a polynomial of lower order, and, as remarked above, complete the proof by invoking the inductive hypothesis. The proof is a combination of a probabilistic approach involving Hoeffding inequality and external approximation of dense varieties (which is in turn based on dependent random choice) and a careful algebraic manipulation of identities that hold for multilinear forms.
\end{itemize}

\noindent\textbf{Paper organization.} The paper organization follows the outline above. We begin with a preliminary section, where we gather useful auxiliary results. After that, there are four sections, each corresponding to a step of the proof. Finally, we devote a section to putting all ingredients together and completing the proof of Theorem~\ref{main_thm}.\\ 

\noindent\textbf{Acknowledgements.} This research was supported by the Ministry of Science, Technological Development and Innovation of the Republic of Serbia through the Mathematical Institute of the Serbian Academy of Sciences and Arts, and by the Science Fund of the Republic of Serbia, Grant No.\ 11143, \textit{Approximate Algebraic Structures of Higher Order: Theory, Quantitative Aspects and Applications} - A-PLUS.

\section{Preliminaries}

\noindent\textbf{Notation.} We frequently use sequence shorthand notation such as $X_{[k]}$, $x_{[k]}$, $X_I$, $x_I$. Namely, when $X_1, \dots, X_k$ are sets, we write $X_{[k]}$ for the product $X_1 \tdt X_k$. More generally, given a non-empty index set $I \subseteq [k]$, we write $X_I = \prod_{i \in I} X_i$. For sequences of elements, we write $x_{[k]}$ for the tuple $(x_1, \dots, x_k)$. Finally, $x_I$ stands for the sub-sequence $(x_i : i \in I)$.\\

For a polynomial $Q$ of degree $k$ we denote by $\mathcal{L}_Q$ its \textit{associated symmetric multilinear form}. Using the \textit{discrete additive derivative} notation, defined as $\Delta_a F(x) = F(x + a) - F(x)$ whenever $F$ is a map between abelian groups $G$ and $H$ and $a \in G$ is the \textit{shift}, we may define $\mathcal{L}_Q$ as $\mathcal{L}_Q(a_1, \dots, a_k) = (k!)^{-1}\Delta_{a_1} \Delta_{a_2} \dots \Delta_{a_k}Q(x)$, noting that the latter expression is independent of $x$. Note that if $Q = Q_{\on{hom}} + Q_{\on{low}}$ where $Q_{\on{hom}}$ is homogeneous of degree $k$ and $Q_{\on{low}}$ is of degree at most $k - 1$, then we have the identity $Q_{\on{hom}}(x) = \mathcal{L}_Q(x, x, \dots, x)$.\\

We also write $\mathbb{D}$ for the unit disk and frequently use averaging notation $\ex_{x \in X}$ as a shorthand for $\frac{1}{|X|} \sum_{x \in X}$. All implicit constants in the paper can be assumed to depend on $k$ unless otherwise specified.\\

\noindent\textbf{Preliminaries on multilinear forms.} When $U_1, \dots, U_k$ are vector spaces over $\mathbb{F}_p$, a \textit{multilinear form} is a map $\psi : U_1\tdt U_k \to \mathbb{F}_p$ which is linear in each variable separately.\\
\indent To stress the role of $x_1, x_2, \dots,  x_k$ as arguments, we write $\on{bias}_{x_1 \in U_1, x_2 \in U_2, \dots, x_k \in U_k} \psi$ for 
\[\ex_{x_1 \in U_1, x_2 \in U_2, \dots, x_k \in U_k} \chi\Big(\psi(x_1, x_2, \dots, x_k)\Big),\]
where $\chi$ is the chosen additive character. This notation is useful when $\psi$ is defined implicitly, such as $(x, y) \mapsto \beta(a x, by)$, when $a,b \in G_m$ are fixed and $x,y \in G_{n - m}$.\\

By a \textit{multilinear variety} we mean a variety which is the zero set of several forms, each of which is multilinear but in a subset of variables. If all forms depend on all variables, we say that the variety is \textit{strictly-multilinear}. If a variety is defined by at most $r$ forms, we say that it is of \textit{codimension at most $r$}.\\

The most basic fact about multilinear varieties of bounded codimension is that they are necessarily dense.

\begin{lemma}[Bounded codimension implies density]\label{bddcodimbound}
    Let $W \subseteq U_{[k]}$ be a multilinear variety of codimension $r$. Then $|W| \geq p^{-r}|U_{[k]}|$.
\end{lemma}

\begin{proof} We prove the claim by induction on $k$, the base case $k = 1$ being a simple linear-algebraic fact. Suppose thus that $k > 1$ and that the claim holds for fewer variables.\\
    \indent Let $W$ be defined by multilinear forms $\alpha_i : U_{I_i} \to \mathbb{F}_p$, for $i \in [r]$. Hence, we may write $W = W_1 \cap (W_2 \times U_k)$, where $W_1 = \{x_{[k]} \in U_{[k]} : (\forall i \in [r])\,\, k \in I_i \implies \alpha_i(x_{I_i}) = 0\}$ and $W_2 = \{x_{[k-1]} \in U_{[k-1]} : (\forall i \in [r]) \,\,k \notin I_i \implies \alpha_i(x_{I_i}) = 0\}$. Thus, for some $r_1 + r_2 \leq r$, the variety $W_1$ has codimension at most $r_1$ and the variety $W_2$ has codimension at most $r_2$. Note that for each $x_{[k-1]} \in W_2$, the set of $x_k \in U_k$ such that $x_{[k]} \in W$ is a subspace of $U_k$ of codimension at most $r_1$. Hence, for each $x_{[k-1]} \in W_2$, we have at least $p^{-r_1}|U_k|$ elements $x_k \in U_k$ with $x_{[k]} \in W$. By induction hypothesis, $|W_2| \geq p^{-r_2}|U_{[k-1]}|$, so the claim follows.
\end{proof}

We need two further elementary observations regarding multilinear forms and their restrictions. Firstly, we show that restrictions do not decrease bias.

\begin{lemma}\label{formsRestrictionLemma1}
    Let $V_{[k]}$ be vector spaces and let $U_i \leq V_i$ for each $i \in [k]$. Let $\alpha : V_{[k]} \to \mathbb{F}_p$ be a multilinear form and define $\beta : U_{[k]} \to \mathbb{F}_p$ by restricting $\alpha$. Then $\on{bias} \beta \geq \on{bias} \alpha$.
\end{lemma}

\begin{proof}
    Let $T_i \leq V_i$ be such that $V_i = T_i \oplus U_i$. Then, after expansion and the triangle inequality
    \begin{align*}\on{bias} \alpha =& \exx_{x_{[k]} \in V_{[k]}} \chi(\alpha(x_{[k]}))\\
    = & \exx_{u_{[k]} \in U_{[k]}, t_{[k]} \in T_{[k]}} \chi(\alpha(u_1 + t_1, \dots, u_k + t_k)) = \exx_{t_{[k]} \in T_{[k]}} \Big(\exx_{u_{[k]} \in U_{[k]}} \chi\Big(\alpha(u_{[k]}) + \alpha(u_1,t_{[2,k]}) + \dots + \alpha(t_{[k]})\Big)\Big)\\
    \leq & \exx_{t_{[k]} \in T_{[k]}} \Big|\exx_{u_{[k]} \in U_{[k]}} \chi\Big(\alpha(u_{[k]}) + \alpha(u_1,t_{[2,k]}) + \dots + \alpha(t_{[k]})\Big)\Big|.\end{align*}
    Note that $\alpha(u_{[k]}) = \beta(u_{[k]})$. But, by an inequality of Lovett (Lemma 2.1. in~\cite{Lovett}), we have for each $t_{[k]} \in T_{[k]}$ 
    \[\Big|\exx_{u_{[k]} \in U_{[k]}} \chi\Big(\beta(u_{[k]}) + \alpha(u_1,t_{[2,k]}) + \dots + \alpha(t_{[k]})\Big)\Big| \leq \on{bias} \beta.\qedhere\]
\end{proof}

Secondly, we show that restrictions to low-codimensional subspaces do not increase bias dramatically. For a finite-dimensional vector space $V$ over $\mathbb{F}_p$, we write $\cdot : V \times V \to \mathbb{F}_p$ for a dot product, which is a symmetric, non-degenerate bilinear form. Note that linear forms $V \to \mathbb{F}_p$ are precisely the maps of the form $x \mapsto a \cdot x$ for a fixed $a \in V$.

\begin{lemma}\label{formsRestrictionLemma2}
    Let $V_{[k]}$ be vector spaces and let $U_i \leq V_i$ for each $i \in [k]$. Let $\alpha : V_{[k]} \to \mathbb{F}_p$ be a multilinear form and define $\beta : U_{[k]} \to \mathbb{F}_p$ by restricting $\alpha$. Then $\on{bias} \alpha \geq p^{-\sum_{i \in [k]}( \dim V_i - \dim U_i)} \on{bias} \beta$.
\end{lemma}

\begin{proof}
    As in the previous proof, let $T_i \leq V_i$ be such that $V_i = T_i \oplus U_i$. Let $A : V_{[k-1]} \to V_k$ be the multilinear map such that $\alpha(x_{[k]}) = A(x_{[k-1]}) \cdot x_k$. Then
    \begin{align*}\on{bias} \alpha = &\exx_{x_{[k-1]} \in V_{[k-1]}} \Big(\exx_{x_k \in V_k} \chi(A(x_{[k-1]}) \cdot x_k)\Big) = \exx_{x_{[k-1]} \in V_{[k-1]}} \id(A(x_{[k-1]}) = 0) \\
    = &\frac{|\{x_{[k-1]} \in V_{[k-1]} : A(x_{[k-1]}) = 0\}|}{|V_{[k-1]}|}.\end{align*}
    On the other hand, for $\beta$, recalling the notation $S^\perp = \{x: (\forall s \in S) x \cdot s = 0\}$, we get
    \begin{align*}\on{bias} \beta = &\exx_{x_{[k-1]} \in U_{[k-1]}} \Big(\exx_{x_k \in U_k} \chi(A(x_{[k-1]}) \cdot x_k)\Big) = \exx_{x_{[k-1]} \in U_{[k-1]}} \id(A(x_{[k-1]}) \in U_{k}^\perp) \\
    =& \frac{|\{x_{[k-1]} \in U_{[k-1]} : A(x_{[k-1]}) \in U_k^\perp\}|}{|U_{[k-1]}|}.\end{align*}

    Note that for each $x_{[k-2]}$, the map on $V_{k-1}$ given by $y_{k-1} \mapsto A(x_{[k-2]}, y_{k-1})$ is linear, so zero is taken the largest number of times. Hence 
    \[|\{x_{[k-1]} \in U_{[k-1]} : A(x_{[k-1]}) \in U_k^\perp\}| \leq |U_k^\perp| |\{x_{[k-1]} \in U_{[k-1]} : A(x_{[k-1]}) = 0 \}|.\]
    It follows that 
    \begin{align*}\on{bias} \beta \leq & |U_k^\perp| \frac{|\{x_{[k-1]} \in U_{[k-1]} : A(x_{[k-1]}) = 0 \}|}{|U_{[k-1]}|} = \frac{|V_{[k-1]}|}{|U_{[k-1]}|} |U_k^\perp| \frac{|\{x_{[k-1]} \in U_{[k-1]} : A(x_{[k-1]}) = 0 \}|}{|V_{[k-1]}|}\\
    = & p^{\sum_{i \in [k]} (\dim V_i - \dim U_i)}\on{bias} \alpha.\qedhere\end{align*}
\end{proof}

Recently, the first-named author proved that dense multilinear varieties contain multilinear subvarieties of bounded codimension with optimal bounds up to multiplicative constant. 

\begin{theorem}[Weak structural result with linear bounds, Theorem 2 in~\cite{LukaOptimalWeak}] \label{weakinverse} For each $d \in \mathbb{N}$ there exists a positive integer $K = K(k)$ (independent of $p$) for which the following holds. Let $V \subseteq U_1\tdt U_k$ be a multilinear variety of density $c$. Then $V$ contains a multilinear variety of codimension at most $K (\log_{p} c^{-1} + 1)$. \end{theorem}

An auxiliary result from that paper will also be used here. It says that very dense subsets of multilinear varieties of bounded codimension fill in the gaps after directional convolutions. Recall that the \textit{convolution in direction} $i$ is the operator that acts on functions $f: U_{[k]} \to \mathbb{R}$ as $\bigconv{i} f(x_{[k]}) = \exx_{y_i \in U_i} f(x_{[i-1]}, y_i + x_i, x_{[i+1, k]}) f(x_{[i-1]}, y_i, x_{[i+1, k]})$.

\begin{lemma}[Directional convolutions of varieties, Lemma 4 in~\cite{LukaOptimalWeak}]\label{convvar} Let $W \subseteq U_{[k]}$ be a multilinear variety of codimension $r$. Suppose that $B \subseteq W$ is a subset of size $|B| \leq 2^{-2k} p^{-kr} |U_{[k]}|$. Then
\[\bigconv{k}\bigconv{k-1} \dots \bigconv{1} \id_{W \setminus B}(x_{[k]}) > 0\]
holds for all $x_{[k]} \in W$.  
\end{lemma}

We also make use of a result of Lovett, which says that strictly-multilinear varieties are positively-correlated. A closely related formulation is that bias of multilinear forms is subadditive. Due to their equivalence, we combine these two statements into a single lemma.

\begin{lemma}[Lovett, Theorem 1.5 and Claim 1.6 in \cite{Lovett}]\label{lovettcorrelation}\phantom{a}\\
    \indent\textbf{(i)} Let $V$ and $W$ be strictly-multilinear varieties inside $U_{[k]}$. Then $|V \cap W| |U_{[k]}| \geq |V||W|$.\\
    \indent \textbf{(ii)} Let $\alpha, \beta : U_{[k]} \to \mathbb{F}_p$ be multilinear forms. Then $\on{bias}(\alpha + \beta) \geq \on{bias}(\alpha)\on{bias}(\beta)$.
\end{lemma}

Apart from these results on multilinear forms, we also need the following inequality.

\begin{lemma}\label{qineq}
 Let $x_1\le \ldots \le x_k$ and $y_1\le \ldots \le y_k$ be two sequences. Suppose that if $y_j=y_l$ then $x_j=x_l$. Then for any permutation $\sigma\in \on{Sym}_k$ such that there exists $j\in [k]$ for which $x_{\sigma(j)}\neq x_j$, we have that $$\sum_{j=1}^k(x_j+y_j)^2 > \sum_{j=1}^k(x_j+y_{\sigma(j)})^2$$   
\end{lemma}

\begin{proof}
    Note that the square terms on both sides cancel out so we just have to prove that $\sum_{j=1}^k x_jy_j>\sum_{j=1}^k x_jy_{\sigma(j)}$. Note that \begin{align}\label{rearrangesum}\sum_{j=1}^k x_jy_j\ge \sum_{j=1}^k x_jy_{\sigma(j)}
    \end{align}by the rearrangement inequality so suppose that equality holds in (\ref{rearrangesum}). Let $j_1$ be such that $x_{\sigma(j_1)}\neq x_{j_1}$. Let $(j_1j_2\ldots j_l)$ be a cycle in $\sigma$ and let $j_{l+1}=j_1$. We may assume without loss of generality that $x_{\sigma(j_1)}> x_{j_1}$. Let $a\le l$ be the smallest positive integer such that $x_{j_{a+1}}\le x_{j_1}$. Now $x_{j_a} > x_{j_1}$ but $x_{j_{a+1}}<x_{j_2}$ so $y_{j_{a+1}}<y_{j_2}$. Now let $\sigma'$ be the permutation in $\on{Sym}_k$ such that $\sigma'(j)=\sigma(j)$ for all $j\neq j_1,j_a$ and $\sigma'(j_a)=\sigma(j_1),\sigma'(j_1)=\sigma(j_a)$, We obtain that $$\sum_{j = 1}^k x_jy_{\sigma'(j)}-\sum_{j = 1}^k x_jy_{\sigma(j)}=(x_{j_a}-x_{j_1})(y_{j_2}-y_{j_{a+1}})>0.$$ But now equality cannot hold in \eqref{rearrangesum} which is a contradiction. This proves the lemma.
\end{proof}

\section{Reduction to algebraic problem}

In this section, we carry out the first step of the proof. Recall that $\mathcal{L}_Q$ denotes the unique symmetric $k$-linear map such that $Q_{\on{hom}}(f) = \mathcal{L}_Q(f,f,\ldots ,f)$ where $Q_{\on{hom}}$ is the homogeneous degree-$k$ part of $Q$. In this step, we rely on Proposition 15 in~\cite{BL}, which we now recall. In the paper~\cite{BL}, the map $\Phi$ appearing in the proposition is assumed to be a phase of a quadratic. However, the proof does not use this fact, and works for arbitrary functions, so we state it in full generality.

\begin{proposition}\label{bienvenulestep1}
    Let $c > 0$ and let $\Phi : G_n \to \mathbb{D}$ be a function. Suppose that $\Big|\sum_{f\in G_n}\mu(f)\Phi(f)\Big| \geq c p^n$. Then either there exists $m \leq n/9$ so that
    \[\exx_{a \in A_m} \Big|\exx_{x \in G_{n- m}} \Phi(ax)\Big|^2 \geq \frac{c^2}{16n^5},\]
    or there is a $m \in [n/18, 17n/18]$ such that
    \[\exx_{x, x' \in G_{n - m}} \exx_{a,a' \in A_m} \Phi(ax)\overline{\Phi(a'x)}\,\overline{\Phi(ax')}\Phi(a'x') \geq \frac{c^4}{256n^{10}}.\]
\end{proposition}

The next proposition encapsulates the first step of the proof and follows easily from Proposition~\ref{bienvenulestep1} and the Gowers-Cauchy-Schwarz inequality.

\begin{proposition}\label{passtomultilemma}
    There exists a constant $C_1 \geq 1$ depending only on $k$ such that the following holds. Let $0 \leq m_0 \leq n/18$ be an integer. Suppose that  $\Big|\sum_{f\in G_n}\mu(f)\chi(Q(f))\Big| \geq c p^n$. Then
    \[\on{bias} \mathcal{L}_Q \geq p^{-km_0}\Big(\frac{c}{2n}\Big)^{C_1}\]
    or there is some $m \in [m_0, 17n/18]$ such that there are at least $(c/2n)^{C_1} p^{k m}$ $k$-tuples $a_{[k]}\in G_{m}^k$ such that the multilinear map 
    \[x_{[k]} \rightarrow \sum_{\pi \in \on{Sym}_k}\mathcal{L}_Q(a_1x_{\pi(1)},\ldots, a_k x_{\pi(k)})\]
    defined on $G_{n-m}^k$ has bias at least $(c/2n)^{C_1}$. 
\end{proposition} 

\begin{proof}
    Apply Proposition~\ref{bienvenulestep1} with $\Phi(f) = \chi(Q(f))$. Let us immediately observe that the first case of the proposition implies a conclusion similar to the second one. Namely, if for some $m \leq n/9$ we have
    \begin{align*}\frac{c^2}{16n^5}\leq &\exx_{a \in A_m} \Big|\exx_{x \in G_{n- m }} \chi(Q(ax))\Big|^2 = \exx_{a \in A_m} \exx_{x, x' \in G_{n- m }} \chi(Q(ax) - Q(ax'))\\
    \leq & \exx_{x, x' \in G_{n- m }} \Big|\exx_{a \in A_m}  \chi(Q(ax) - Q(ax'))\Big|\\
    \leq & \sqrt{\exx_{x, x' \in G_{n- m }} \Big|\exx_{a \in A_m}  \chi(Q(ax) - Q(ax'))\Big|^2}\\
    =&\sqrt{\exx_{x, x' \in G_{n- m }}\exx_{a,a' \in A_m}  \chi(Q(ax) - Q(ax') - Q(a'x) + Q(a'x'))}, \end{align*}
    where we used the Cauchy-Schwarz inequality in the penultimate line.\\

    Thus, we only need to discuss the possibility that $m  < m_0$ or to proceed with the second case.\\

    \noindent\textbf{Small $m$ case.} The outcome of Proposition in particular gives some $a \in A_m$ such that $\Big|\ex_{x \in G_{n- m}} \chi(Q(ax))\Big| \geq  \frac{c^2}{16n^5}$. As $a$ is monic, we have that $U = a \cdot G_{n-m}$ is a subspace of $G_n$ of codimension $m$. Thus $\Big|\ex_{x \in U} \chi(Q(x))\Big| \geq  \frac{c^2}{16n^5}$. Squaring, expanding and making a change of variables, we obtain 
    \[\exx_{x,b \in U} \chi(\Delta_b Q(x)) \geq \Big(\frac{c^2}{16n^5}\Big)^2.\]

    Repeating Cacuhy-Schwarz inequalities with change of variables, we obtain
    \[\on{bias} \mathcal{L}_Q|_{U \times \dots \times U} = \exx_{x, b_1, \dots, b_k \in U}  \chi(\Delta_{b_1, \dots, b_k} Q(x)) \geq \Big(\frac{c^2}{16n^5}\Big)^{2^k}.\]

    By Lemma~\ref{formsRestrictionLemma2}, we have 
    \[\on{bias} \mathcal{L}_Q|_{U \times \dots \times U} \leq p^{km} \on{bias} \mathcal{L}_Q\]
    completing this part of the proof.\\

    \noindent\textbf{Second outcome.} Hence, we may assume that we have $m_0 \leq m \leq 17n/18$ such that  
    \[\exx_{x, x' \in G_{n- m }}\exx_{a,a' \in A_m}  \chi(Q(ax) - Q(ax') - Q(a'x) + Q(a'x')) \geq \frac{c^4}{256n^{10}}.\]

    We introduce variables $a_1, a_2, \dots, a_k$ ranging over $G_m$ and $x_1, \dots, x_k$ ranging over $G_{n-m}$. Let us rename $a$ and $x$ to $a_0$ and $x_0$, and make a change of variables and replace $a'$ by $a_0 + a_1 + \ldots + a_k$ and $x' = x_0 + x_1 + \ldots + x_k$. Hence, the expression above becomes

    \begin{align} &\exx_{\ssk{x_0, x_1, \dots, x_k \in G_{n- m}\\a_0,a_1, \dots, a_k \in A_m}} \chi\Big(Q\Big((a_0 + a_1 + \dots + a_k)(x_0 + x_1 + \dots + x_k)\Big) - Q\Big(a_0(x_0 + x_1 + \dots + x_k)\Big) \label{step1longEQN}\\
    &\hspace{5cm}- Q((a_0 + a_1 + \dots + a_k)x_0) + Q(a_0x_0)\Big) \geq \frac{c^4}{256n^{10}}.\nonumber\end{align}

    Recall that $Q(f) = Q_{\on{hom}}(f) + Q_{\on{low}}(f)$, where $Q_{\on{hom}}$ is the homogeneous part of degree $k$ and $Q_{\on{low}}$ is the part of $Q$ of degree at most $k-1$, and $Q_{\on{hom}}(f) = \mathcal{L}_Q(f,f,\dots,f)$. Hence
    \[Q_{\on{hom}}\Big((a_0 + a_1 + \dots + a_k)(x_0 + x_1 + \dots + x_k)\Big) = \mathcal{L}_Q\Big(\sum_{i,j = 0}^k a_ix_j, \dots, \sum_{i,j = 0}^k a_ix_j\Big) = \sum_{i_1, \dots, i_k, j_1, \dots, j_k = 0}^k  \mathcal{L}_Q(a_{i_1}x_{j_1}, \dots, a_{i_k}x_{j_k}).\]

    Hence, in expression~\eqref{step1longEQN}, the only terms that involve all variables $a_1, \dots, a_k, x_1, \dots, x_k$ are precisely $\mathcal{L}_Q(a_{i_1}x_{j_1}, \dots, a_{i_k}x_{j_k})$, when both $(i_1, \dots, i_k)$ and $(j_1, \dots, j_k)$ are permutations of $[k]$. Note also that for such indices $(a_{[k]}, x_{[k]}) \mapsto \mathcal{L}_Q(a_{i_1}x_{j_1}, \dots, a_{i_k}x_{j_k})$ is a multilinear form in $2k$ variables. Applying the Gowers-Cauchy-Schwarz inequality for these variables implies that 
    \[\exx_{a_1, \dots, a_k \in G_m, x_1, \dots, x_k \in G_{n - m}} \chi\Big(\sum_{\pi, \sigma \in \on{Sym}_k} \mathcal{L}_Q(a_{\pi(1)} x_{\sigma(1)}, \dots, a_{\pi(k)} x_{\sigma(k)})\Big) \geq \Big( \frac{c^4}{256n^{10}}\Big)^{2^{2k}}.\]
    Since $\mathcal{L}_Q$ is symmetric and $\on{gcd}(k!, p) = 1$, we may simplify the expression above and have $\pi = \on{id}$ only. Since for each fixed $a_1, \dots, a_k$, the expression 
    \[\exx_{x_1, \dots, x_k \in G_{n - m}} \chi\Big(\sum_{\pi \in \on{Sym}_k}\mathcal{L}_Q(a_1x_{\pi(1)},\ldots, a_k x_{\pi(k)})\Big)\]
    is the bias of a multilinear form and hence a non-negative real that is at most 1, the claim follows by averaging.
\end{proof}

\vspace{\baselineskip}

\section{Finding a variety of biased forms}

In this section, we prove a general result on biased sub-forms of a given multilinear form, stated as Theorem~\ref{varietyofbiasedmaps} below, which we then apply in the study of the M\"obius function, to complete the second step of our outline.\\
\indent Let $U_1, \dots, U_k, V_1, \dots, V_\ell$ be finite-dimensional vector spaces over $\mathbb{F}_p$ and let $\alpha : U_1 \tdt U_k \times V_1 \tdt V_\ell  \to \mathbb{F}_p$ be a multilinear form. For $x_{[k]} \in U_{[k]}$, we write $\alpha_{x_{[k]}} : V_{[\ell]} \to \mathbb{F}_p$ for the multilinear form $\alpha_{x_{[k]}}(y_{[\ell]}) = \alpha(x_{[k]}, y_{[\ell]})$.

\begin{theorem}\label{varietyofbiasedmaps}
      Let $K = K(k)$ be the constant from Theorem~\ref{weakinverse}.  Let $\alpha : U_{[k]} \times V_{[\ell]} \to \mathbb{F}_p$ be a multilinear form. For any $\delta > 0$ define $S_\delta = \{x_{[k]} \in U_{[k]} : \on{bias}(\alpha_{x_{[k]}}) \geq \delta\}$.\\
        \indent Suppose that  for some $c > 0$ we have $|S_c| \geq c|U_{[k]}|$. Then $S_{\tilde{c}}$ contains a multilinear variety of codimension at most $K(2 \log_p c^{-1} + 2)$, where $\tilde{c} = 2^{-2^{k+2}k}p^{-2^{k+1}k K} c^{2^{k+2}kK}$.
\end{theorem}

\vspace{\baselineskip}

\noindent\textbf{Remark.} It is natural to compare this to various fibers theorems for mulitlinear varieties, for example Theorem 3.8 in~\cite{KazhZieFibres}, Theorem 38 in~\cite{FMulti} and Theorem 8 in~\cite{LukaPrank}. However, all those previous results rely on the partition vs. analytic rank problem and cannot give the strong bounds that we need.\\

In the proof, we use the following interpretation of the bias.\\

\begin{claim}\label{biasinterpretation}
    Let $\psi : W_{[k]} \to \mathbb{F}_p$ be a multilinear form. Then $\on{bias} \psi \geq c$ is equivalent to there being at least $c |W_1||W_2| \dots |W_{k-1}|$ choices of $(k - 1)$-tuples $x_{[k-1]} \in W_{[k-1]}$ such that $\psi(x_{[k-1]}, y) = 0$ for all $y \in W_k$.
\end{claim}

\begin{proof}
    By the definition of bias, we have 
    
    \[\on{bias} \psi =\exx_{x_1 \in W_1, \dots, x_{k-1} \in W_{k-1}, y \in W_k} \chi\Big(\psi(x_1, \dots, x_{k-1}, y)\Big) \geq c.\]
    
    We know that there exists a multilinear map $\Psi : W_{[k-1]} \to W_k$, such that $\psi(x_{[k-1]},y) = \Psi(x_{[k-1]}) \cdot y$ holds for all $x_1 \in W_1, \dots, x_{k-1} \in W_{k-1}, y \in W_k$. Thus, $\on{bias} \psi \geq c$ is equivalent to
    \[c \leq \exx_{x_1 \in W_1, \dots, x_{k-1} \in W_{k-1}} \Big(\exx_{y \in W_k} \omega^{\Psi(x_{[k-1]}) \cdot y}\Big).\]
    But, the inner average is $\id(\Psi(x_{[k-1]}) = 0)$. Hence, this is equivalent to there being at least $c|W_{[k-1]}|$ choices of $x_{[k-1]} \in W_{[k-1]}$ such that $\Psi(x_{[k-1]}) = 0$, and this is equivalent to having $c|W_{[k-1]}|$ choices of $(k-1)$-tuples $x_{[k-1]} \in W_{[k-1]}$ such that $\psi(x_{[k-1]}, y) = 0$ for all $y \in W_k$.
\end{proof}

\begin{proof}[Proof of Theorem~\ref{varietyofbiasedmaps}]
Let $a_{[k]} \in S_c$. Then
\[\on{bias}_{b_1 \in V_1, b_2 \in V_2,\dots, b_\ell \in V_\ell} \alpha(a_{[k]},b_{[\ell]}) \geq c\]
so by Claim~\ref{biasinterpretation} we have at least $c|V_{[\ell-1]}|$ choices of $b_{[\ell-1]} \in V_{[\ell-1]}$ such that
\begin{equation}
    (\forall b_\ell \in V_\ell)\,\,\alpha(a_{[k]}, b_{[\ell]}) = 0.\label{vanishingequation}
\end{equation}

Let $T \subseteq U_{[k]} \times V_{[\ell-1]}$ be the collection of all $(a_{[k]}, b_{[\ell-1]})$ such that~\eqref{vanishingequation} holds. Since $|S_c| \geq c|U_{[k]}|$, $|T| \geq c^2 |U_{[k]}||V_{[\ell-1]}|$.\\

We now apply a dependent random choice argument. Let $\varepsilon = 2^{-2k}p^{-2k K} c^{2kK}$ and $c' = \frac{c^2 \varepsilon}{2}$. Pick $y_1 \in V_1, \dots, y_{\ell-1} \in V_{\ell-1}$ uniformly and independently at random. Set $A = \{a_{[k]} \in U_{[k]} : (a_{[k]}, y_{[\ell-1]}) \in T\}$ and let $B$ be the set of all $a_{[k]} \in A$ such that $T_{a_{[k]}} := \{z_{[\ell-1]} \in V_{[\ell-1]} : (a_{[k]}, z_{[\ell-1]}) \in T\}$ has size at most $c'|V_{[\ell-1]}|$. We think of points in $B$ as bad points (hence the notation $B$).\\
\indent Observe that for any $a_{[k]} \in U_{[k]}$ the probability of $a_{[k]} \in A$ is precisely $\frac{|T_{a_{[k]}}|}{|V_{[\ell-1]}|}$. By the linearity of expectation we have
\begin{align*}\exx_{y_{[\ell-1]}} \Big(|A| - \varepsilon^{-1}|B|\Big) = &\sum_{a_{[k]} \in U_{[k]}} \mathbb{P}(a_{[k]} \in A) - \varepsilon^{-1} \sum_{a_{[k]} \in U_{[k]}} \mathbb{P}(a_{[k]} \in B) \\
= &\sum_{a_{[k]} \in U_{[k]}} \frac{|T_{a_{[k]}}|}{V_{[\ell-1]}} - \varepsilon^{-1} \sum_{a_{[k]} \in U_{[k]}} \frac{|T_{a_{[k]}}|}{|V_{[\ell-1]}|} \id\Big(|T_{a_{[k]}}| \leq c' V_{[\ell-1]}\Big)\\
\geq & \frac{|T|}{V_{[\ell-1]}} - \varepsilon^{-1} c' |U_{[k]}| \geq \frac{c^2}{2}|U_{[k]}|.\end{align*}

Hence, there exists a choice of $y_{[\ell-1]}$ which gives $|A| \geq \frac{c^2}{2}|U_{[k]}|$ and $|B| \leq \varepsilon |A|$.\\

Note that $A$ itself is a multilinear variety. Namely,
\begin{align*}A = & \{a_{[k]} \in U_{[k]} : (a_{[k]}, y_{[\ell-1]}) \in T\} \\
= &\Big\{a_{[k]} \in U_{[k]} : (\forall z_\ell \in V_\ell)\,\,\alpha(a_{[k]}, y_{[\ell-1]}, z_\ell) = 0 \Big\}\\
= & \bigcap_{z_\ell \in V_\ell} \Big\{a_{[k]} \in U_{[k]} : \alpha(a_{[k]}, y_{[\ell-1]}, z_\ell) = 0\Big\}.
\end{align*}
Since $|A| \geq \frac{c^2}{2}|U_{[k]}|$, by Theorem~\ref{weakinverse}, it contains a multilinear variety $W$ of codimension $r$ for some $r \leq K(2 \log_p c^{-1} + 2)$. Recall that $|B| \leq \varepsilon |A|$. In particular, by our choice of $\varepsilon$,
\[|W \cap B|\leq |B| \leq \varepsilon|U_{[k]}| \leq  |U_{[k]}| \leq 2^{-2k} p^{-kr} |U_{[k]}|,\]
so by Lemma~\ref{convvar}, we have
\begin{equation}\bigconv{k}\bigconv{k-1} \dots \bigconv{1} \id_{W \setminus B}(a_{[k]}) > 0\label{convgenvar}\end{equation}
for all $a_{[k]} \in W$. The following claim allows us to reach the desired conclusion.

\begin{claim}
    Suppose that $a_{[k]}$ and $b_{[k]}$ differ only at coordinate $i$. Define $c_j = a_j = b_j$ for $j \not= i$, and $c_i = a_i - b_i$. If $|T_{a_{[k]}}|, |T_{b_{[k]}}| \geq \delta |V_{[\ell-1]}|$, then $|T_{c_{[k]}}| \geq \delta^2 |V_{[\ell-1]}|$.
\end{claim}

\begin{proof}
    Similarly to above, for given $a_{[k]}$, we have that the set of all $z_{[\ell-1]} \in V_{[\ell-1]}$ such that~\eqref{vanishingequation} holds is a strictly multilinear variety. It is of density $\delta$ by assumptions. Analogous conclusion holds for $b_{[k]}$. By Lemma~\ref{lovettcorrelation}, the intersection of these strictly multilinear varieties has density at least $\delta^2$. Let $z_{[\ell-1]}$ be any point in the intersection. Then
    \[\alpha(c_{[k]}, z_{[\ell]}) = \alpha(a_{[k]}, z_{[\ell]}) - \alpha(b_{[k]}, z_{[\ell]}) = 0\]
    holds for all $z_\ell \in V_\ell$, so $|T_{c_{[k]}}| \geq \delta^2 |V_{[\ell - 1]}|$.
\end{proof}

As a consequence of the claim and~\eqref{convgenvar}, we have $|T_{a_{[k]}}| \geq {c'}^{2^k}|V_{[\ell - 1]}|$ for all $a_{[k]} \in W$. By Claim~\ref{biasinterpretation}, we obtain $a_{[k]} \in S_{\tilde{c}}$ for $\tilde{c} = {c'}^{2^k}$.
\end{proof}

\textbf{Application to M\"obius uniformity.} The theorem above allows us to quickly complete the second step of the proof of our main result. The assumptions in the following proposition stem from the conclusion of Proposition~\ref{passtomultilemma}.

\begin{proposition}\label{mobiuscorrvariety}
    There exists a constant $C_{2} \geq 1$ depending only on $k$ such that the following holds. Suppose that the multilinear form
    \[x_{[k]} \mapsto \sum_{\pi \in \on{Sym}_k} \mathcal{L}_Q(a_1x_{\pi(1)}, a_2 x_{\pi(2)},\dots, a_k x_{\pi(k)})\]
    defined on $(G_{n - m})^k$, has bias at least $c$ for at least $c |G_{m}|^k$ choices of $a_{[k]} \in (G_{m})^k$. Then there exists a multilinear variety $W \subseteq (G_{m})^k$ of codimension at most $C_{2} \log_p(2c^{-1})$ such that for every $a_{[k]} \in W$, we have
    \[\on{bias}_{x_1, \dots, x_k \in G_{n-m}} \sum_{\pi \in \on{Sym}_d}  \mathcal{L}_Q(a_1x_{\pi(1)}, a_2 x_{\pi(2)},\dots, a_k x_{\pi(k)}) \geq (c/2)^{C_{2}}.\]
\end{proposition}

\begin{proof}
    The map $\alpha: G_{m}^k\times G_{n-m}^k \to \mathbb{F}_p$, given by
    \[(a_1, \dots, a_k, x_1, x_2,\dots, x_k) \mapsto \sum_{\pi \in \on{Sym}_k}  \mathcal{L}_Q(a_1x_{\pi(1)}, a_2 x_{\pi(2)},\dots, a_k x_{\pi(k)})\]
    is multilinear in all of its $2k$ variables. Thus Theorem~\ref{varietyofbiasedmaps} applies to give the desired multilinear variety $W$.
\end{proof}

\section{Obtaining large bias in $\mathcal{L}_Q$}\label{LQhighbias}

The proof in this section uses a method inspired by the one in~\cite{BL} which Bienvenu and L\^{e} used to complete their argument. Starting from the conclusion of Proposition~\ref{mobiuscorrvariety} the goal is to deduce that $\mathcal{L}_Q$ has large bias. Bienvenu and L\^{e} considered ranks of matrices and used a combinatorial partitioning argument to control the rank. In our case, we work with bias as the measure of the structure of multilinear forms, so we phrase our argument analytically.\\

Let us state the main claim of this section.

\begin{proposition}\label{psitolqlargebiasstep}
    There exists a constant $C_{3} \geq 1$ depending only on $k$ such that the following holds. Let $d$ and $m$ be such that $d \leq m \leq \frac{17n}{18}$ and $d \leq \frac{n}{18(k+2)}$ and suppose that we have a multilinear variety $W \subseteq G_m^k$ of codimension at most $r$ such that for every $a_{[k]} \in W$, the multilinear form $\psi_{a_{[k]}} : G_{n-m} ^k \to \mathbb{F}_p$ defined by
    \[\psi_{a_{[k]}}(x_{[k]}) = \sum_{\sigma \in \on{Sym}_d} \mathcal{L}_Q (a_1x_{\sigma(1)}, a_2 x_{\sigma(2)},\dots, a_k x_{\sigma(k)})\]
    has bias at least $c$. If $d > k r \Big(\frac{n}{d}\Big)^k$, then $\on{bias} \mathcal{L}_Q \geq p^{-C_3 d} (c/2)^{C_3^{(n/d)^{C_3}}}$.
\end{proposition}

\noindent\textbf{Remark.} With a more careful analysis at the end of the proof, we may be more efficient and decrease the exponent $C_3^{(n/d)^{C_3}}$. However, as we eventually apply Proposition~\ref{psitolqlargebiasstep} with the ratio $n/d$ being a constant depending only on $k$, this change would only affect the term $\varepsilon_k$ in the main result (Theorem~\ref{main_thm}) at the cost of complicating the notation.

\begin{proof}
    Let $g = \lfloor m/d\rfloor$, and $s = g -1 + \lfloor (n-m)/d\rfloor$. In particular, these choices imply $s \leq n/d$ and $s - g + 1 = \lfloor (n-m)/d\rfloor$, further giving $d(s - g + 1) \leq n - m$. Note that, since $m \leq \frac{17n}{18}$, we have $d \leq \frac{n - m}{k + 2}$, so $k \leq \frac{n-m}{d} - 2 \leq \lfloor (n-m)/d\rfloor - 1 = s - g$, which will be important later (see Claim~\ref{aellpropsclaim}).\\
    
    In this proof, we shall first find a suitable polynomial $w$ of degree at most $d$, and consider the basis of $G_n$ given by $1, t, \dots, t^{\deg w - 1}, w, wt, \dots, wt^{n - 1- \deg w}$. Using this basis, we may find linear maps $\pi_{\on{res}} : G_n \to G_{\deg w}$, $\pi_0 : G_n \to G_d, \dots, \pi_{s - 1} : G_n \to G_d$ and $\pi_{\on{over}} : G_n \to w t^{sd} \cdot G_{n - \deg w - sd}$ such that for each $x \in G_n$ we have
    \[x = \pi_{\on{res}}(x) + w \pi_0(x) + wt^d \pi_1(x) + \dots + wt^{(s-1)d} \pi_{s-1}(x) + \pi_{\on{over}}(x).\]
    (To explain the subscripts, we remark that $\pi_{\on{res}}$ arises as the residue in division by $w$ and $\pi_{\on{over}}$ is the 'overflowing' part.)\\
    
    Let us also write $\pi_{\on{main}}(x) = w \pi_0(x) + wt^d \pi_1(x) + \dots + wt^{(s-1)d} \pi_{s-1}(x)$, so $x = \pi_{\on{res}}(x) + \pi_{\on{main}}(x) + \pi_{\on{over}}(x)$. Using these linear maps, we may express $\mathcal{L}_Q$ using forms defined on smaller domains. Namely

    \begin{align*}\mathcal{L}_Q(x_1, \dots, x_k) = & \mathcal{L}_Q(\pi_{\on{res}}(x_1) + \pi_{\on{main}}(x_1) + \pi_{\on{over}}(x_1), \dots, \pi_{\on{res}}(x_k) + \pi_{\on{main}}(x_k) + \pi_{\on{over}}(x_k)) \\
    = &\sum_{\phi_1, \dots, \phi_k \in \{\pi_{\on{res}}, \pi_{\on{main}}, \pi_{\on{over}}\}} \mathcal{L}_Q(\phi_1(x_1), \dots, \phi_k(x_k)),\end{align*}

    where in the last line we used multilinearity of $\mathcal{L}_Q$ to reduce to a sum over all choices of forms $\phi_1, \dots, \phi_k$ among $\pi_{\on{res}}, \pi_{\on{main}}, \pi_{\on{over}}$.\\

    By Lemma~\ref{lovettcorrelation}, it suffices to get lower bounds on the bias of $\mathcal{L}_Q(\phi_1(x_1), \dots, \phi_k(x_k))$ to get the lower bound on $\on{bias} \mathcal{L}_Q$.\\
    \indent We observe that the bias is large when some $\phi_i \not= \pi_{\on{main}}$. Namely, we have
    \[\on{bias}\mathcal{L}_Q \circ (\phi_1, \dots, \phi_k) = \sum_{v \in \on{Im} \phi_i} \exx_{x_i} \id(\phi_i(x_i) = v) \exx_{x_{[k] \setminus \{i\}}} \chi \circ \mathcal{L}_Q(\phi_1(x_1), \dots, \phi_{i-1}(x_{i-1}), v, \phi_{i+1}(x_{i+1}), \dots, \phi_{k}(x_{k})).\]

    For fixed $v$, the expression
    \[\exx_{x_{[k] \setminus \{i\}}} \chi \circ \mathcal{L}_Q(\phi_1(x_1), \dots, \phi_{i-1}(x_{i-1}), v, \phi_{i+1}(x_{i+1}), \dots, \phi_{k}(x_{k}))\]
    is the bias of a multilinear form in $k-1$ variables, so it is a non-negative real and it is 1 when $v = 0$. So the contribution from $v= 0$ is at least $\ex_{x_i} \id(\phi_i(x_i) = 0) = |\on{ker}\phi_i|/|G_n| = 1/|\on{Im} \phi_i| \geq p^{-2d}$, since ranks of $\pi_{\on{over}}$ and $\pi_{\on{res}}$ are at most $2d$. By Lemma~\ref{lovettcorrelation}, we thus obtain
    \begin{equation}\label{smallpartsremovallqbias}\on{bias} \mathcal{L}_Q \geq p^{-2\cdot3^k d} \on{bias}\mathcal{L}_Q \circ (\pi_{\on{main}}, \dots, \pi_{\on{main}}).\end{equation}

    \textbf{Bias contribution from the main part.} Again, using multilinearity, we have
    \begin{equation}\mathcal{L}_Q(\pi_{\on{main}}(x_1), \dots, \pi_{\on{main}}(x_k)) = \sum_{i_1, \dots, i_k \in [0, s-1]^k} \mathcal{L}_Q(wt^{i_1d} \pi_{i_1}(x_1), \dots, wt^{i_kd} \pi_{i_k}(x_k)).\label{lqmaindecomposition}\end{equation}

    To control the bias of the multilinear forms on the right-hand-side, we make a choice of adequate $w \in G_d$ first. Recall that we have a special multilinear variety $W$ of codimension at most $r$.

    \begin{claim}
        There exists a non-zero $w\in G_{d}$ such that $\on{bias} (\psi_{wt^{i_1d}, \ldots, wt^{i_kd}}) \geq c$ for any choice $i_{[k]} \in [0, g-1]^k$.     
    \end{claim}

    \begin{proof}
        It suffices to find some non-zero $w \in G_d$ such that $(wt^{i_1d}, \ldots, wt^{i_kd}) \in W$ holds for all $i_{[k]}\in [0, g-1]^k$. Note that $W$ is defined as the zero set of at most $r$ multilinear forms, namely, we have some $\alpha_j: G_m^{J_j} \to \mathbb{F}_p$, for $j \in [r]$, where $J_j$ is a non-empty subset of $[k]$ (but possibly proper) and $W$ is the set of all $x_{[k]} \in G_m^k$ such that for each $j\in [r]$ we have $\alpha_j(x_{J_j}) = 0$. Hence, we need to find $w \in G_{d}$ such that $\alpha_j(wt^{i_
        \ell d}: \ell \in J_j) = 0$ for all choices of $j \in [r], i_{[k]}\in [0,g-1]^k$. In other words, $w$ is a solution of a system of at most $r s^k$ polynomial equations of degree at most $k$. As $\dim G_{d} = d > k \Big(\frac{n}{d}\Big)^k r \geq k s^k r$, Chevalley-Warning theorem applies, so the number of such $w \in G_{m}$ is divisible by $p$. Since $w = 0$ is one such solution, there exists another one.
    \end{proof}

    Take $w$ provided by the claim above. We shall use the large bias of $\psi_{wt^{i_1d}, \ldots, wt^{i_kd}}$ to deduce the large bias of $\mathcal{L}_Q \circ (\pi_{\on{main}}, \dots, \pi_{\on{main}})$. Note that for any $i_1, \dots, i_k \in [0,g-1]$, we have
    \[\on{bias} \psi_{wt^{i_1d}, \ldots, wt^{i_kd}} = \exx_{z_1, \dots, z_k \in G_{n-m}} \chi\circ\Big(\sum_{\sigma \in \on{Sym}_k} \mathcal{L}_Q(wt^{i_1d} z_{\sigma(1)}, \dots, wt^{i_k d} z_{\sigma(k)})\Big).\]
    Since $\mathcal{L}_Q$ is symmetric, we have $\mathcal{L}_Q(wt^{i_1d} z_{\sigma(1)}, \dots, wt^{i_k d} z_{\sigma(k)}) = \mathcal{L}_Q(wt^{i_{\sigma^{-1}(1)}d} z_{1}, \dots, wt^{i_{\sigma^{-1}(k)} d} z_{k})$, so
    \[\on{bias} \psi_{wt^{i_1d}, \ldots, wt^{i_kd}} = \exx_{z_1, \dots, z_k \in G_{n- m}} \chi\circ\Big(\sum_{\sigma \in \on{Sym}_k} \mathcal{L}_Q(wt^{i_{\sigma^{-1}(1)}d} z_{1}, \dots, wt^{i_{\sigma^{-1}(k)} d} z_{k})\Big).\]

    Observe that by Lemma~\ref{formsRestrictionLemma1}, we may restrict the averages to any product of subspaces of $G_{n-m}$ without decreasing the bias. This will be crucial in order to complete the proof.\\

    For $i_{[k]} \in [0,s-1]^k$, let $\mathsf{f}(i_{[k]})$ be the number of permutations $\sigma \in \on{Sym}_k$ which fix the sequence, namely $i_{\sigma^{-1}(1)} = i_1, \dots, i_{\sigma^{-1}(k)} = i_k$. If $a_1, \dots, a_\ell$ are multiplicities of values appearing in the sequence, we have $\mathsf{f}(i_{[k]}) = a_1! \cdots a_\ell!$, which is non-zero modulo $p$.\\

    Note that, as $x$ ranges over $G_n$, the $s$-tuple $(\pi_0(x), \dots, \pi_{s-1}(x))$ ranges uniformly over $G_d^s$. Going back to~\eqref{lqmaindecomposition}, and introducing variables $y_{j,j'} \in G_d$ for $j \in [k], j' \in [0,s-1]$, we have 
    \begin{align*}
        \on{bias} \mathcal{L}_Q \circ (\pi_{\on{main}}, \dots,\pi_{\on{main}}) = \exx_{y_{1,0}, \dots, y_{k, s-1} \in G_d} \chi\circ\Big(\sum_{i_1, \dots, i_k \in [0, s-1]} \mathcal{L}_Q(wt^{i_1d} y_{1, i_1}, \dots, wt^{i_kd} y_{k, i_k})\Big).
    \end{align*}

    Note that we may view $\mathcal{L}_Q(wt^{i_1d} y_{1, i_1}, \dots, wt^{i_kd} y_{k, i_k})$ as a multilinear form in $k$ variables where the $i$\tss{th} variable is $(y_{i,0}, \dots, y_{i, s-1})$. By Lemma \ref{lovettcorrelation} it is sufficient to show that these maps have high biases. Furthermore, for each $i_{[k]}\in [0,s-1]^k$ denote by $\mathcal{L}_{i_{[k]}}$ the multilinear form on $G_d^k$ given by $x_{[k]} \mapsto  \mathcal{L}_Q(wt^{i_1d} x_1, \dots, wt^{i_kd} x_k)$.
    As the bias of the form $\mathcal{L}_Q(wt^{i_1d} y_{1, i_1}, \dots, wt^{i_kd} y_{k, i_k})$ equals $\on{bias} \mathcal{L}_{i_{[k]}}$, we further reduce the problem to showing that all $\mathcal{L}_{i_{[k]}}$ have high bias.\\

    Another piece of notation we want to introduce is a permutation of a linear map. Given a sequence of natural numbers $e_1, \dots, e_k$, a multilinear form $\mathcal{T}: G_{e_1}\times \ldots \times G_{e_k}\rightarrow \mathbb{F}_p$ and a permutation $\sigma \in \on{Sym}_k$ we denote by $\mathcal{T}^{\sigma}$ the multilinear form on $G_{e_{\sigma(1)}}\tdt G_{e_{\sigma(k)}}$ defined by $\mathcal{T}^{\sigma}(x_1,\ldots, x_k)=\mathcal{T}(x_{\sigma(1)},\ldots ,x_{\sigma(k)})$. Notice that $\on{bias} \mathcal{T}^{\sigma}=\on{bias} \mathcal{T}$. Note that we may only sum linear maps with the same dimensions in each direction. For example, if $e_1 = e_2$ then we can still define $\mathcal{T}+\mathcal{T}^{(12)}$.\\

    \noindent\textbf{Brief overview.} The main strategy of our proof will be to list all the $k$-tuples of $[0,s-1]^k$ in a particular order $v_1, v_2, \ldots, v_{s^k}$ so that
    \begin{itemize}
        \item the first $s$ elements have large bias, namely $\on{bias} \mathcal{L}_{v_i}\ge c$ for all $i\le s$,
        \item for the remaining elements, i.e. when $i > s$, the form $\mathcal{L}_{v_i}$ can be expressed as a sum of at most $k!$ multilinear forms each of which is either a map of bias at least $c$ or a permutation of some form among $\mathcal{L}_{v_{1}}, \dots, \mathcal{L}_{v_{i-1}}$.
    \end{itemize} 
    Notice that, once we achieve this, then by Lemma~\ref{lovettcorrelation} and induction, we will have that $\on{bias} \mathcal{L}_{v_i} \ge c^{k!^i}$. Applying Lemma~\ref{lovettcorrelation} another time gives the lower bound of $c^{s^{k}k!^{s^{k}}}$ on $\on{bias} \mathcal{L}_Q \circ (\pi_{\on{main}}, \dots,\pi_{\on{main}})$, which we combine with~\eqref{smallpartsremovallqbias}, to finish the proof.\\
    
    \noindent\textbf{Completing the proof.} Define a partial order $\prec $ on $[0,s-1]^k$ to satisfy $j_{[k]} \prec \ell_{[k]}$ if $j_1^2+\ldots +j_k^2 < \ell_1^2+\ldots \ell_k^2$. For the first $s$ elements $v_1,\ldots v_s$, we simply set $v_i = (i-1,i-1,\dots,i-1)$, $i \in [s]$. Let us check that $\mathcal{L}_{v_i}$ for these elements have large bias.

    \begin{claim}
        For $i \in [s]$, we have $\on{bias} \mathcal{L}_{v_i} \geq c$. 
    \end{claim}

    \begin{proof}
        When $i \in [0, g-1]$, then $\mathcal{L}_{v_i}$ is the restriction of the map $\psi_{wt^{(i-1)d}, \dots, wt^{(i-1)d}}$ to $G_d^k$, so by Lemma~\ref{formsRestrictionLemma1}, we get the desired bias. For $i \in [g, s]$, we notice that $\mathcal{L}_{v_i}$ is the restriction of the map $\psi_{wt^{(g-1)d}, \dots, wt^{(g-1)d}}$ to $(t^{(i - g)d}\cdot G_d)^k$, so we are again done by Lemma~\ref{formsRestrictionLemma1}.
    \end{proof}
    
    Let $v_{s+1},v_{s+2}\ldots v_{s^k}$ be the rest of the elements in $[0,s-1]^k$ sorted in any order  such that if $v_i \prec v_j$ then $i < j$, i.e. the $k$-tuples with smaller sum of squares appear earlier. In the rest of the argument, we show that for each $i > s$ we can express $\mathcal{L}_{v_i}$ in terms of a multilinear form of bias at least $c$ and at most $k!-1$ permutations of $\mathcal{L}_{v_j}$ where $j < i$.\\

Let $i > s$ and let $v_i=(j_1,\ldots j_k)$. Since $\mathcal{L}_Q$ is symmetric and thus for any permuation $\sigma \in \on{Sym}_k$ we have $\on{bias} \mathcal{L}_{j_1, \dots, j_k} = \on{bias} \mathcal{L}_{j_{\sigma(1)}, \dots, j_{\sigma(k)}}$, without loss of generality we may assume that $j_1\le j_2\le \ldots \le j_k$. Define another sequence of indices $\ell_1, \ell_2,\ldots ,\ell_k \in [0,s-1]$ inductively as follows. Let $\ell_1 = \min(j_1,g-1)$. Then, for an index $1\leq \iota \leq k-1$, if $j_{\iota + 1} > j_\iota + 1$ then let $\ell_{\iota+1} = \min (j_{\iota+1}-(j_\iota - \ell_\iota) - 1, g-1)$, and if $j_{\iota+1} \le j_\iota + 1$ then let $\ell_{\iota+1}= \ell_\iota$. Now let $a_\iota = j_\iota-\ell_\iota$ for each $\iota \in [k]$. Notice that the $\ell_{[k]}$ and $a_{[k]}$ satisfy the following properties.

\begin{claim}\label{aellpropsclaim} We have
\begin{itemize}
    \item[\textbf{(i)}] $\ell_1 \le \ell_2\le \ldots \le \ell_k$,
    \item[\textbf{(ii)}] $a_1 \le a_2\le \ldots \le a_k$,
    \item[\textbf{(iii)}] $a_\iota \leq s - g$ for every $\iota \in [k]$,
    \item[\textbf{(iv)}] if $\ell_{\iota_1} > \ell_{\iota_2}$ then $a_{\iota_1} > a_{\iota_2}$.
\end{itemize}
\end{claim}

We remark that we use that fact $k \leq s - g$ in the proof of the third property.

\begin{proof}
    For the first property, firstly note that if some $\ell_\iota$ equals $g-1$, then by the definition of $\ell_{[k]}$, we have $\ell_{\iota + 1} = \dots = \ell_k = g - 1$. So, we just need to consider the case when $\ell_{\iota + 1} < g - 1$. In the second case of the definition, we have $\ell_{\iota+1}= \ell_\iota$. In the first case, if $\ell_{\iota + 1} < g - 1$, then $\ell_{\iota+1}=j_{\iota+1}-(j_\iota - \ell_\iota) - 1 = \ell_\iota + (j_{\iota+1}-j_\iota -1) > \ell_\iota $ as the first case applies only when $j_{\iota+1}-j_\iota \geq 2$.\\
    \indent The second property follows from the inequality $\ell_{\iota + 1} \leq j_{\iota+1}-(j_\iota - \ell_\iota) - 1$, which holds in the first case of inductive definition, and $\ell_{\iota +1} = \ell_\iota$, which holds in the second case, combined with the fact that $j_{[k]}$ is increasing.\\
    \indent For \textbf{(iii)}, by the second property, it suffices to show that $a_k \leq s - g$. If some $\ell_\iota$ equals $g-1$, then $\ell_{\iota + 1} = \dots = \ell_k = g-1$, so $a_k \leq s-g$, as $j_k \leq s- 1$. On the other hand, if $\ell_\iota < g-1$ for all $\iota \in [k]$, then we have that $a_1 = j_1 - \ell_1 = 0$ and $a_{\iota + 1} \leq a_{\iota} + 1$ at each step, so $a_k \leq k-1$. As $k \leq s-g$, we are done.\\
    \indent Finally, as both $\ell_{[k]}$ and $a_{[k]}$ are increasing, we only need to show that $\ell_{\iota + 1} > \ell_{\iota}$ implies $a_{\iota + 1} > a_{\iota}$. But the first inequality only happens when in the first case of definition, so
    \[a_{\iota + 1} = j_{\iota + 1} - \ell_{\iota + 1} \geq j_{\iota+1} - \Big( j_{\iota+1}-(j_\iota - \ell_\iota) - 1\Big) = (j_\iota - \ell_\iota) + 1 = a_\iota + 1.\qedhere\]
\end{proof}

\vspace{\baselineskip}

Now consider the restriction $\mathcal{H}$ of 
\[
\psi_{wt^{\ell_1d},\ldots ,wt^{\ell_kd}}(z_1,\ldots z_k)=\sum_{\sigma \in \on{Sym}_k} \mathcal{L}_Q(wt^{\ell_1d}z_{\sigma(1)},\ldots wt^{\ell_kd}z_{\sigma(k)})
\]
to the subspace $t^{a_1d}G_d\times \cdots \times t^{a_kd}G_d$. We remark that property \textbf{(iii)} of the claim above implies that $d(a_\iota + 1)\leq d(s - g + 1) \leq n-m$, so $t^{a_\iota d}G_d \leq G_{n-m}$ and the restriction $\mathcal{H}$ is well-defined.\\
\indent On the left hand side is a multilinear form of bias at least $c$. Now we examine to what the terms on the right hand side of the expression above correspond for each permutation $\sigma$.\\

First suppose that $\sigma$ is a permutation such that $\ell_{\sigma(\iota)} = \ell_\iota$ for all $\iota$. Then $a_\iota + \ell_{\sigma(\iota)} = a_\iota + \ell_\iota = j_\iota$ for each $\iota$, so the term contributing to $\mathcal{H}$ on the right hand side is precisely $\mathcal{L}_{v_i}$. The number of such permutations $\sigma$ is $\mathsf{f}(j_{[k]}) \neq 0$. \\

Now suppose that $\sigma $ is a permutation such that for some $i$ we have $x_{\sigma(i)}\neq x_i$. The term contributing to $\mathcal{H}$ on the right hand side for $\sigma$ is 
\[\mathcal{T}_\sigma(z_{[k]}) := \mathcal{L}_Q(wt^{\ell_1d}z_{\sigma(1)},\ldots wt^{\ell_kd}z_{\sigma(k)})\]
which is a permutation of 
\[\mathcal{L}_Q(wt^{\ell_{\sigma^{-1}(1)}d}z_1,\ldots wt^{\ell_{\sigma^{-1}(k)}d}z_{k})\]
where we recall that $z_\iota$ ranges over $t^{a_\iota d}G_d$. Hence, this multilinear form is precisely $\mathcal{L}_{v'}$ for the sequence $v' = (\ell_{\sigma^{-1}(1)} + a_1, \dots, \ell_{\sigma^{-1}(k)} + a_k)$. We now apply Lemma \ref{qineq} to obtain that $v' \prec v_i$, so $v'$ appears before $v_i$ in our ordering of elements of $[0, s-1]^k$, i.e. $v' =v_{i'}$ for some $i'<i$.\\

Thus  we obtain that 
\[
\mathsf{f}(i_{[k]})\mathcal{L}_{v_i} = \mathcal{H} - \sum_{\sigma \in S'} \mathcal{T}_{\sigma}
\]

where $S'$ is the set of permutations $\sigma$ for which there is at least one index $\iota$ such that $\ell_{\sigma(\iota)} \neq \ell_\iota$ and each $\mathcal{T}_{\sigma}$ is a suitable multilinear form produced above, being a permutation of some $\mathcal{L}_{v'}$ for $i' < i$. Hence by above we have that $\mathcal{L}_{v_i}$ is a sum of a multilinear form of bias at least $c$ and at most $k!$ multilinear forms which are permutations of preceding multilinear forms. This completes the proof.
\end{proof}

\section{Case of biased polynomials}

In this section, we treat the special case of correlation of M\"obius function with phases of biased polynomials.

\begin{proposition}\label{red_to_alg}
    There exists a constant $C_4 \geq 1$ depending only on $k$ such that the following holds. Let $Q : G_n \to \mathbb{F}_p$ be a polynomial of degree $k$ such that $\on{bias} \mathcal{L}_Q \geq c$ and suppose that $|\sum_{f\in G_n}\mu(f)\chi(Q(f))| \geq c|G_n|$. Then there is a polynomial $\tilde{Q}$ of degree at most $k-1$ such that
    \[|\sum_{f\in G_n}\mu(f)\chi(\tilde{Q}(f))| \geq (c/2)^{C_4}|G_n|.\]
\end{proposition}

The most of this section will be devoted to approximations of phases of biased polynomials by lower-order phases. Our main result is Theorem~\ref{biased_poly_approx}, which will then easily imply the proposition above.

\subsection{Approximation of biased polynomials}

Throughout the subsection, $U, U_1, \dots, U_k$ will be finite-dimensional vector spaces over $\mathbb{F}_p$.

\begin{theorem}\label{biased_poly_approx}
    Let $Q: U \to \mathbb{F}_p$ be a polynomial of degree $k$, where $k < p$. Suppose that $\on{bias} \mathcal{L}_Q \geq c$. Let $\varepsilon > 0$. Then there exist $m \leq (\varepsilon c/2)^{-O(1)}$, coefficients $c_1, \dots, c_m \in \mathbb{C}$ such that $\sum_{i \in [m]} |c_i| \leq c^{-1}$ and polynomials $R_1, \dots, R_m : G \to \mathbb{F}_p$ of degree at most $k-1$ such that 
    \[\Big\| \chi \circ Q - \sum_{i \in [m]} c_i \chi \circ R_i\Big\|_{L^2} \leq \varepsilon.\]
\end{theorem}

We prove the analogous result for multilinear forms and deduce the theorem from that. A \textit{multiaffine form} is a map $\alpha : U_{[k]} \to \mathbb{F}_p$ which is affine in each variable separately. Equivalently, it is a sum of multilinear forms $\alpha(x_{[k]}) = \sum_{I \subseteq [k]} \beta_I(x_I)$, each $\beta_I$ depending on the subset of variables $x_I$. Note that $\beta_{[k]}$ is unique due to the identity
\[\beta(d_{[k]}) = \sum_{I \subseteq [k]} (-1)^{k - |I|} \alpha((x_i + d_i)_{i \in I}, (x_i)_{i \in I^c}).\]
We refer to $\beta_{[k]}$ as the \textit{multilinear part} of $\alpha$. In our case, vanishing multilinear part means that $\alpha$ is of lower-order.

\begin{theorem}\label{biased_ml_form_approx}
    Let $\alpha : U_{[k]} \to \mathbb{F}_p$ be a multilinear form with $\on{bias}\alpha \geq c$. Let $\varepsilon > 0$. Then there exist
    $m \leq 2^{25}\varepsilon^{-14}c^{-14}$, coefficients $c_1, \dots, c_m \in \mathbb{C}$ such that $\sum_{i \in [m]} |c_i| \leq c^{-1}$, and multiaffine forms $\lambda_1, \dots, \lambda_m : U_{[k]} \to \mathbb{F}_p$ with vanishing multilinear part such that 
    \[\Big\| \chi \circ \alpha - \sum_{i \in [m]} c_i \chi \circ \lambda_i\Big\|_{L^2} \leq \varepsilon.\]
\end{theorem}

\begin{proof}
    Let $A: U_{[k-1]} \to U_k$ be the multilinear map such that $A(x_{[k-1]}) \cdot x_k = \alpha(x_{[k]})$. Write $Z = \{x_{[k-1]} \in U_{[k-1]} : A(x_{[k-1]}) = 0\}$, which has density at least $c$ by assumptions.\\

    Let $t_{[k-1]} \in U_{[k-1]}$ be arbitrary. Define the multiaffine map $L_{t_{[k-1]}} : U_{[k-1]} \to U_k$ by
    \[L_{t_{[k-1]}}(x_{[k-1]}) = \sum_{I \subsetneq [k-1]} (-1)^{k - |I|} A(x_I, t_{[k-1] \setminus I}).\]
    We remark that the multilinear part of $L_{t_{[k-1]}}$ vanishes since $I = [k-1]$ is not included in the sum.

    \begin{claim}\label{relationLandA}
        Whenever $x_{[k-1]} \in t_{[k-1]} + Z = \{(t+z)_{[k-1]} : z_{[k-1]} \in Z\}$, we have
        \[A(x_{[k-1]}) = L_{t_{[k-1]}}(x_{[k-1]}).\]
    \end{claim}

    \noindent\textbf{Remark.} The notation $(t+z)_{[k-1]}$ is a shorthand for the sequence $t_1 + z_1, \dots, t_{k-1} + z_{k-1}$ and $t_{[k-1]} + Z$ is simply a translate of $Z$ by $t_{[k-1]}$.\\

    \begin{proof}
        Take $z_{[k-1]} \in Z$ such that $x_i  = z_i + t_i$. Then
        \begin{align*}A(x_{[k-1]}) - L_{t_{[k-1]}}(x_{[k-1]}) = &A(x_{[k-1]})- \sum_{I \subsetneq [k-1]} (-1)^{k - |I|} A(x_I, t_{[k-1] \setminus I}) \\
        = &A(x_{[k-1]}) + \sum_{I \subsetneq [k-1]} (-1)^{k - 1 - |I|} A(x_I, t_{[k-1] \setminus I})\\
        = &\sum_{I \subseteq [k-1]} (-1)^{k - 1 - |I|} A(x_I, t_{[k-1] \setminus I})\\
        = &A(x_1 - t_1, \dots, x_{k-1} - t_{k-1}) = A(z_{[k-1]}) = 0.\qedhere\end{align*}
    \end{proof}

    \noindent\textbf{First approximation.} Let $m$ be a positive integer to be chosen later and let $t^{(1)}_{[k-1]}, \dots, t^{(m)}_{[k-1]} \in U_{[k-1]}$ be chosen uniformly and independently at random. By Hoeffding's inequality for $m$ independent random variables taking values in $[0,1]$, for each $x_{[k-1]} \in U_{[k-1]}$, the probability that 
    \begin{equation} \label{prbabilitybound} \Big||\{i \in [m]: x_{[k-1]} \in t^{(i)}_{[k-1]} + Z\}| - \frac{Z}{|U_{[k-1]}|}m\Big| \leq \frac{\varepsilon}{4} \frac{Z}{|U_{[k-1]}|}m \end{equation}
    is at least $1 - 2\exp(- 2^{-3}\varepsilon^2 c^2 m)$. Taking $m = 2^{7}\varepsilon^{-4} c^{-4}$ and using the elementary inequality $\exp(-M) = \frac{1}{\exp(M)} \leq \frac{1}{1 + M}\leq \frac{1}{M}$ for all $M > 0$, we get probability at least
    \[1 - 2\exp(- 2^{4}\varepsilon^{-2} c^{-2}) \geq 1 - \frac{c^2\varepsilon^2}{8}.\]

    \begin{claim}
        There exists a choice of $t^{(1)}_{[k-1]}, \dots, t^{(m)}_{[k-1]}$ such that the function 
        \[f(x_{[k]}) = \frac{|U_{[k-1]}|}{m |Z|} \sum_{i \in [m]} \id_{t^{(i)}_{[k-1]} + Z}(x_{[k-1]}) \chi\Big(L_{t^{(i)}_{[k-1]}}(x_{[k-1]}) \cdot x_k\Big)\]
        satisfies $\| \chi \circ \alpha - f\|_{L^2} \leq \varepsilon/2$.
    \end{claim}

    \begin{proof}
        Due to the probability bound above, we have a choice of $t^{(1)}_{[k-1]}, \dots, t^{(m)}_{[k-1]}$ such that~\eqref{prbabilitybound} holds on a set $X \subseteq U_{[k-1]}$ of size $|X| \geq (1 - \frac{c^2\varepsilon^2}{8})|U_{[k-1]}|$. Also, owing to Claim~\ref{relationLandA}, $\id_{t^{(i)}_{[k-1]} + Z}(x_{[k-1]}) \chi\Big(L_{t^{(i)}_{[k-1]}}(x_{[k-1]}) \cdot x_k\Big)$ is $\chi\Big(A(x_{[k-1]}) \cdot x_k\Big)$ when $x_{[k-1]} \in t^{(i)}_{[k-1]} + Z$, and 0 otherwise. Hence, we deduce that $|\chi \circ \alpha(x_{[k]}) - f(x_{[k]})| \leq \frac{|U_{[k-1]}|}{m |Z|} \frac{\varepsilon}{4} \frac{Z}{|U_{[k-1]}|}m = \frac{\varepsilon}{4}$ when $x_{[k-1]} \in X$, and $|\chi \circ \alpha(x_{[k]}) - f(x_{[k]})| \leq \frac{|U_{[k-1]}|}{|Z|} \leq c^{-1}$ for trivial reasons otherwise. Hence
        \begin{align*}\| \chi \circ \alpha - f\|_{L^2}^2 = &\frac{1}{|U_{[k-1]}|} \sum_{x_{[k-1]} \in U_{[k-1]}} |\chi \circ \alpha(x_{[k]}) - f(x_{[k]})|^2\\
        \leq &\frac{1}{|U_{[k-1]}|}\sum_{x_{[k-1]} \in X} \Big(\frac{\varepsilon}{4}\Big)^2 + \frac{1}{|U_{[k-1]}|}\sum_{x_{[k-1]} \in U_{[k-1]} \setminus X}c^{-2} \leq \varepsilon^2/4.\qedhere\end{align*}
    \end{proof}
    
    \vspace{\baselineskip}

    Let $t^{(1)}_{[k-1]}, \dots, t^{(m)}_{[k-1]}$ and $f$ be as in the claim.\\

    \vspace{\baselineskip}

    \noindent\textbf{Second approximation.} We may externally approximate $Z$ by a strictly-multilinear variety $V \subseteq U_{[k-1]}$ of codimension $s \leq 2\log_p (4\varepsilon^{-1} c^{-1} m)$ with $|V \setminus Z| \leq \Big(\frac{\varepsilon c}{4m}\Big)^2|U_{[k-1]}|$. Note that the $L^2$ norm of the function mapping $x_{[k]}$ to  
    \[\id_{t^{(i)}_{[k-1]} + V}(x_{[k-1]}) \chi\Big(L_{t^{(i)}_{[k-1]}}(x_{[k-1]}) \cdot x_k\Big) - \id_{t^{(i)}_{[k-1]} + Z}(x_{[k-1]}) \chi\Big(L_{t^{(i)}_{[k-1]}}(x_{[k-1]}) \cdot x_k\Big)\]
    is at most $\sqrt{|V \setminus Z|/|U_{[k-1]}|} \leq \frac{\varepsilon c}{4m}$.\\
    \indent Hence, defining 
    \[g(x_{[k]}) = \frac{|U_{[k-1]}|}{m |Z|} \sum_{i \in [m]} \id_{t^{(i)}_{[k-1]} + V}(x_{[k-1]}) \chi\Big(L_{t^{(i)}_{[k-1]}}(x_{[k-1]}) \cdot x_k\Big),\]
    by triangle inequality, we have $\| \chi \circ \alpha - g\|_{L^2} \leq \varepsilon$.\\

    \noindent\textbf{Final form of approximation.} Let $V = \beta^{(-1)}(0)$ for a multilinear map $\beta: U_{[k-1]} \to \mathbb{F}_p^s$. Then we may rewrite $g$ as 
    \begin{align*}g(x_{[k]}) = &\frac{|U_{[k-1]}|}{m |Z|} \sum_{i \in [m]} \id(\beta(x_1 - t^{(i)}_1,\dots,x_{k-1} - t^{(i)}_{k-1}) = 0) \chi\Big(L_{t^{(i)}_{[k-1]}}(x_{[k-1]}) \cdot x_k\Big)\\
    = & \frac{|U_{[k-1]}|}{m |Z|} \sum_{i \in [m]} p^{-s}\sum_{\tau \in \mathbb{F}_p^s} \chi\Big(\tau\cdot \beta(x_1 - t^{(i)}_1,\dots,x_{k-1} - t^{(i)}_{k-1})\Big) \chi\Big(L_{t^{(i)}_{[k-1]}}(x_{[k-1]}) \cdot x_k\Big)\\
    = &  \sum_{i \in [m], \tau \in \mathbb{F}_p^s} \frac{|U_{[k-1]}|}{p^s m |Z|} \chi\Big(\tau\cdot \beta(x_1 - t^{(i)}_1,\dots,x_{k-1} - t^{(i)}_{k-1}) + L_{t^{(i)}_{[k-1]}}(x_{[k-1]}) \cdot x_k\Big)
    \end{align*}
    which is the desired form of the approximation (we stress that we get phases of multiaffine forms with vanishing multilinear part). Note that the coefficients of the multiaffine phases are nonnegative reals, adding to $|U_{[k-1]}| / |Z| \leq c^{-1}$. The total number of terms in the approximation is 
    \[p^s m \leq (4\varepsilon^{-1} c^{-1} m)^2 m \leq 2^{25}\varepsilon^{-14}c^{-14}.\qedhere\]
\end{proof}

\begin{proof}[Proof of Theorem~\ref{biased_poly_approx}]
    Write $Q(x) = Q_k(x) + Q_{\on{low}}(x)$, where $Q_k(x)$ is a homogeneous polynomial consisting only of monomials of degree $k$ and $Q_{\on{low}}(x)$ has degree at most $k-1$. Recall that $\mathcal{L}_Q(a_1, \dots, a_k)$ is the symmetric multilinear form such that $\mathcal{L}_Q(x, \dots, x) = Q_k(x)$ and $\mathcal{L}_Q(a_1, \dots, a_k) = (k!)^{-1} \Delta_{a_1, \dots, a_k} Q(x)$. By assumptions $\on{bias} \mathcal{L}_Q \geq c$.\\

    Theorem~\ref{biased_ml_form_approx} for approximation parameter $\varepsilon$ applies to give
    \[\Big\| \chi \circ \mathcal{L}_Q  - \sum_{i \in [r]} c_i \chi \circ \lambda_i\Big\|_{L^2} \leq \varepsilon,\]
    with $r \leq (2c^{-1}\varepsilon^{-1})^{O(1)}$ and $\sum_{i \in [r]} |c_i| \leq c^{-1}$.\\
    
    By expressing points of $U^k$ as $(t_1 + x, \dots, t_k + x)$, for $t_1, \dots, t_k, x \in U$, we have
    \[\exx_{x, t_1, \dots, t_k} \Big|\chi \circ\mathcal{L}_Q (t_1 + x, \dots, t_k + x) - \sum_{i \in [r]} c_i \chi \circ \lambda_i(t_1 + x, \dots, t_k + x)\Big|^2 \leq \varepsilon^2.\]

    By averaging, there exist $t_1, \dots, t_k$ such that
 \[\exx_{x} \Big|\chi \circ \mathcal{L}_Q (t_1 + x, \dots, t_k + x) - \sum_{i \in [r]} c_i \chi \circ \lambda_i(t_1 + x, \dots, t_k + x)\Big|^2 \leq \varepsilon^2.\]

 Note that $\mathcal{L}_Q(t_1 + x, \dots, t_k + x) = \mathcal{L}_Q(x, x, \dots, x) + Q'(x)$ for some polynomial $Q'$ of degree at most $k - 1$, and $R_i(x) := \lambda_i(t_1 + x, \dots, t_k + x)$ is a polynomial of degree at most $k - 1$ in $x$, as the multilinear part of each $\lambda_i$ vanishes. Hence
 \[\Big\|\chi \circ (Q_k + Q') - \sum_{i \in [r]} c_i \chi \circ R_i\Big\|_{L^2} \leq \varepsilon.\]
 As $Q = Q_k + Q_{\on{low}}$, it follows that $\sum_{i \in [r]} c_i \chi \circ (R_i + Q_{\on{low}} - Q')$ is the desired approximation.
\end{proof}

We may now prove Proposition~\ref{red_to_alg}.

\begin{proof}[Proof of Proposition~\ref{red_to_alg}]
    Approximate $\chi(Q(x))$ by $\sum_{i \in [r]} c_i \chi \circ R_i$ with error $c^2/4$. By the Cauchy-Schwarz inequality, and averaging over $i \in [r]$, we get correlation
    \[\Big|c_i\sum_{f\in G_n}\mu(f)\chi(R_i(f))\Big| \geq (c/2)^{O(1)}|G_n|.\]
    Since the coefficients satisfy the inequality $|c_i| \leq c^{-1}$, we get the conclusion.
\end{proof}

\subsection{Inverse theorem for uniformity norms in case of polynomial phases}

A common application of partition vs. analytic rank problem is the inverse theorem for uniformity norms in case of polynomial phases. Using the approximation result of this section, we may easily deduce polynomial bounds for this problem.

\begin{corollary}
    Let $G = \mathbb{F}_p^n$, $Q : G\to \mathbb{F}_p$ be a polynomial of degree $k$ and assume $k < p$. If $\|\chi \circ Q\|_{\mathsf{U}^k} \geq c$ then there exists a polynomial $P : G \to \mathbb{F}_p$ of degree at most $k-1$ such that
    \[\Big|\exx_{x \in G} \chi(Q(x))\overline{\chi(P(x))}\Big| \geq (c/2)^{O_k(1)}.\]
\end{corollary}

\begin{proof}
     Expanding the definition of the uniformity norm, we have
    \[c^{2^k} \leq \|\chi \circ Q\|_{\mathsf{U}^k}^{2^k} = \exx_{x, a_1, \dots, a_k \in G} \chi\Big(\Delta_{a_1, \dots, a_k} Q(x)\Big) = \on{bias} \mathcal{L}_Q.\]

    Apply Theorem~\ref{biased_poly_approx} with approximation parameter $\varepsilon = 1/2$ to find $m \leq (\varepsilon c/2)^{-O(1)} \leq (c/2)^{-O(1)}$, coefficients $c_1, \dots, c_m \in \mathbb{C}$ such that $\sum_{i \in [m]} |c_i| \leq c^{-2^{k}}$ and polynomials $R_1, \dots, R_m : G \to \mathbb{F}_p$ of degree at most $k-1$ such that, writing $F(x) = \sum_{i \in [m]} c_i \chi \circ R_i(x)$,
    \[\Big\| \chi \circ Q - F\Big\|_{L^2} \leq 1/2.\]
    Since $\|\chi \circ Q\|_{L^2} = 1$, by the Cauchy-Schwarz inequality, we get
    \[\Big|\exx_{x \in G} \chi \circ Q(x) \overline{F(x)}\Big| \geq \Big|\exx_{x \in G} \chi \circ Q(x) \overline{\chi \circ Q(x)}\Big|  - \Big|\exx_{x \in G} \chi \circ Q(x) \overline{(\chi \circ Q(x) - F(x))}\Big| \geq 1 - \|\chi \circ Q\|_{L^2}\|\chi \circ Q - F\|_{L^2} \geq 1/2.\]
    Hence $\Big|\exx_{x \in G} \chi \circ Q(x) \overline{\Big(\sum_{i \in [m]} c_i \chi \circ R_i(x)\Big)}\Big| \geq 1/2$, so by averaging, we find some $i \in [m]$ with
    \[\Big|\exx_{x \in G} \chi \circ Q(x) \overline{\chi \circ R_i(x)\Big)}\Big| \geq 1/(2m) \geq (c/2)^{O(1)}.\qedhere\]
\end{proof}

\section{Putting everything together}

We are ready to prove our main result, Theorem~\ref{main_thm}.

\begin{proof}[Proof of Theorem~\ref{main_thm}]
   We prove the claim by induction on $k$. The base case $k = 1$ was carried out in~\cite{BL}. Recall that $C_1 = C_1(k), C_2 = C_2(k), C_3 = C_3(k)$ and $C_4 = C_4(k)$ are constants from Propositions~\ref{passtomultilemma},~\ref{mobiuscorrvariety},~\ref{psitolqlargebiasstep} and~\ref{red_to_alg} corresponding to the main steps of the proof. Let $\eta_0 = \varepsilon_{k-1}/(20kC_3C_4)$ where $\varepsilon_{k-1}$ is the constant from the inductive hypothesis. Thus $\eta_0$ depends on $k$ only.\\
   
   Suppose that $c = \Big|\sum_{f \in G_n} \mu(f)\chi(Q(f))\Big|$. Define $m_0 = d = \lfloor \eta_0 n\rfloor$ and note that for sufficiently large $n$ in terms of $\eta_0$ only, we have $2d > \eta_0 n$. Apply Proposition~\ref{passtomultilemma} with such a choice of $m_0$. If we obtain the first outcome, we get
   \begin{equation}
       \label{finalsectionFirstOutcome} 
       \on{bias} \mathcal{L}_Q \geq p^{-km_0}\Big(\frac{c}{2n}\Big)^{C_1}.
   \end{equation}

   Let us now focus on the second outcome. Eventually, we will obtain a similar inequality to~\eqref{finalsectionFirstOutcome}.\\
   
   Thus, we have $m_0 \leq m \leq 17n/18$ such that there are at least $(c/2n)^{C_1} p^{k m}$ $k$-tuples $a_{[k]}\in G_{m}^k$ such that the multilinear form $\psi_{a_{[k]}}$ 
    \[x_{[k]} \rightarrow \sum_{\pi \in \on{Sym}_k}\mathcal{L}_Q(a_1x_{\pi(1)},\ldots, a_k x_{\pi(k)})\]
    defined on $G_{n-m}^k$ has bias at least $(c/2n)^{C_1}$.\\

    Apply Proposition~\ref{mobiuscorrvariety} to find a multilinear variety $W \subseteq (G_{n - m})^k$ of codimension $r \leq C_1C_{2} \log_p(4nc^{-1})$ such that for every $a_{[k]} \in W$, we have
    \[\on{bias}_{x_1, \dots, x_k \in G_{n-m}} \sum_{\pi \in \on{Sym}_d}\phi(a_1x_{\pi(1)}, a_2 x_{\pi(2)},\dots, a_k x_{\pi(k)}) \geq (c/4n)^{C_1C_2}.\]

    We may assume that $c \geq 4n p^{ - \frac{1}{k\,2^{k+1} \eta_0^{-(k+1)} C_1C_2}n}$ and $n$ sufficiently large, as otherwise we are done. Thus, 
    \[\log_p(4nc^{-1}) \leq \frac{1}{k\,2^{k+1} \eta_0^{-(k+1)} C_1C_2}n\]
    so, as $n/d < 2\eta_0^{-1}$,
    \[k (n/d)^k r < k \,2^k\eta_0^{-k} r \leq k \,2^k\eta_0^{-k} C_1C_{2} \log_p(4nc^{-1}) \leq \eta_0n/2 < d.\]
    By Proposition~\ref{psitolqlargebiasstep}, as $\eta_0 \leq \frac{1}{18(k+2)}$ so $d \leq \frac{n}{18(k+2)}$, we have $\on{bias} \mathcal{L}_Q\ge p^{-C_3 \eta_0 n}  \Big((c/8n)^{C_1C_2}\Big)^{C_3^{\eta_0^{-C_3}}}$.\\

    Hence, combining the two outcomes and writing $\tilde{C} = C_3^{\eta_0^{-C_3}}$, we have
    \[\on{bias} \mathcal{L}_Q \geq p^{- \eta_0kC_3 n}  (c/8n)^{C_1C_2\tilde{C}}.\]

    Applying Proposition~\ref{red_to_alg}, we see that for some polynomial $\tilde{Q}$ of degree at most $k - 1$
    \[\Big|\exx_{f\in G_n}\mu(f)\chi(\tilde{Q}(f))\Big| \geq p^{- \eta_0kC_3C_4 n} (c/16n)^{C_1C_2\tilde{C}C_4}.\]

    By the inductive hypothesis, we have 
    \[p^{- \eta_0kC_3C_4 n} (c/16n)^{C_1C_2\tilde{C}C_4} \geq p^{-\varepsilon_{k-1}n}.\]
    We choose $\eta_0 = \varepsilon_{k-1}/(20kC_3C_4)$, we have 
    \[c \geq 16 n  p^{-\frac{\varepsilon_{k-1}}{2C_1C_2\tilde{C}C_4}n},\]
    completing the proof.
\end{proof}

\end{document}